\documentclass[11pt]{article}

\usepackage{amstext,amssymb,amsmath,amsbsy}
\usepackage{xcolor}
\usepackage{bm}
\usepackage{hyperref}
\usepackage{amscd}
\usepackage{amsfonts}
\usepackage{indentfirst}
\usepackage{verbatim}
\usepackage{amsmath}
\usepackage{amsthm}
\usepackage{enumerate}
\usepackage{graphicx}
\usepackage{color}
\usepackage[OT1]{fontenc}
\usepackage[latin1]{inputenc}
\usepackage[english]{babel}
\usepackage{amssymb}
\usepackage{subfig}
\usepackage{algorithm}
\usepackage{algpseudocode}
\usepackage{mathtools}

\newcommand{\ds}{\displaystyle}
\newcommand{\bu}{{\bf u}}

\newcommand{\x}{{\bf x}}

\newcommand{\Div}{{\rm div}}

\newtheorem{Theorem}{Theorem}[section]
\newtheorem{Lemma}{Lemma}[section]

\newtheorem{Proposition}{Proposition}[section]

\newtheorem{remark}{Remark}[section]

\newtheorem*{Assumption*}{Assumption}
\newtheorem{Definition}{Definition}[section]
\newtheorem{problem}{Problem}[section]
\newtheorem*{problem*}{Problem}
\numberwithin{equation}{section}

\begin{document}

\title{ Recovery of initial displacement and velocity in anisotropic elastic systems  by the time dimensional reduction method}

\author{
\and
Trong D. Dang\thanks{Faculty of Mathematics and Computer Science, University of Science, Vietnam National University, Ho Chi Minh City, Vietnam, \texttt{ddtrong@hcmus.edu.vn}} 
\and
Chanh V. Le\thanks{Faculty of Mathematics and Computer Science, University of Science, Vietnam National University, Ho Chi Minh City, Vietnam, \texttt{lvchanh@hcmus.edu.vn}}
\and
Khoa D. Luu\thanks{Faculty of Mathematics and Computer Science, University of Science, Vietnam National University, Ho Chi Minh City, Vietnam, \texttt{khoaluu1207@gmail.com}}
\and
Loc H. Nguyen\thanks{Department of Mathematics and Statistics, University of North Carolina at
Charlotte, Charlotte, NC, 28223, USA, \texttt{loc.nguyen@charlotte.edu}, corresponding author} 
}


\date{}
\maketitle

\begin{abstract}
We introduce a time-dimensional reduction method for the inverse source problem in linear elasticity, where the goal is to reconstruct the initial displacement and velocity fields from partial boundary measurements of elastic wave propagation. The key idea is to employ a novel spectral representation in time, using an orthonormal basis composed of Legendre polynomials weighted by exponential functions. This Legendre polynomial-exponential basis enables a stable and accurate decomposition in the time variable, effectively reducing the original space-time inverse problem to a sequence of coupled spatial elasticity systems that no longer depend on time.
These resulting systems are solved using the quasi-reversibility method. On the theoretical side, we establish a convergence theorem ensuring the stability and consistency of the regularized solution obtained by the quasi-reversibility method as the noise level tends to zero. On the computational side, two-dimensional numerical experiments confirm the theory and demonstrate the method's ability to accurately reconstruct both the geometry and amplitude of the initial data, even under substantial measurement noise. The results highlight the effectiveness of the proposed framework as a robust and computationally efficient strategy for inverse elastic source problems.
\end{abstract}



\section{Introduction}

Let $\mathbb{C} : \mathbb{R}^d \to \mathbb{R}^{d \times d \times d \times d}$, with $d \geq 2$, $d \in \mathbb{N}$, denote the \emph{elasticity tensor} characterizing an anisotropic and inhomogeneous elastic medium. 
Let $\mathbf{u} = \begin{bmatrix}u_1 & \dots & u_d\end{bmatrix}^\top: \mathbb{R}^d\times [0,\infty) \to \mathbb{R}^d$ denote the displacement field of the medium. It is governed by the following initial value problem for the elastic wave equation:
\begin{equation}
    \begin{cases}
        \mathbf{u}_{tt}(\mathbf{x}, t) = \operatorname{div} \left( \mathbb{C} : \nabla \mathbf{u}(\mathbf{x}, t) \right), & \mathbf{x} \in \mathbb{R}^d,\ t \in (0, \infty), \\
        \mathbf{u}(\mathbf{x}, 0) = \mathbf{p}(\mathbf{x}), & \mathbf{x} \in \mathbb{R}^d, \\
        \mathbf{u}_t(\mathbf{x}, 0) = \mathbf{q}(\mathbf{x}), & \mathbf{x} \in \mathbb{R}^d,
    \end{cases}
    \label{main_ivp}
\end{equation}
where $\mathbf{p}$ and $\mathbf{q}$ denote the initial displacement and velocity fields, respectively.

The primary goal of this paper is to investigate the following inverse problem, which concerns the recovery of the initial conditions from boundary observations.

\begin{problem}[Recovery of initial data in the elastic wave equation]
Let $T > 0$ be a fixed final time, and let $\Omega \subset \mathbb{R}^d$ be a bounded open domain with smooth boundary $\partial \Omega$. Assume that $\mathbf{u}(\mathbf{x}, t)$ is the solution to the initial value problem \eqref{main_ivp}, and that the following boundary data are available:
\begin{equation}
    \mathbf{f}(\mathbf{x}, t) = \mathbf{u}(\mathbf{x}, t), \quad
    \mathbf{g}(\mathbf{x}, t) = \partial_\nu \mathbf{u}(\mathbf{x}, t), \qquad \text{for } (\mathbf{x}, t) \in \partial \Omega \times [0, T],
    \label{data_ip}
\end{equation}
where $\nu(\mathbf{x})$ denotes the outward unit normal vector to $\partial \Omega$ at the point $\mathbf{x}$, and $\partial_\nu \mathbf{u}$ is the normal derivative of $\mathbf{u}$.
Reconstruct the unknown initial data:
\[
\mathbf{p}(\mathbf{x}) = \mathbf{u}(\mathbf{x}, 0), \quad
\mathbf{q}(\mathbf{x}) = \mathbf{u}_t(\mathbf{x}, 0), \qquad \text{for all } \mathbf{x} \in \Omega.
\]
\label{isp}
\end{problem}

To guarantee the well-posedness of the forward problem and  the uniqueness of the time-reduced model developed in this work, we impose the following assumptions on the elasticity tensor \(\mathbb{C} = \{C_{ijkl}\}_{i,j,k,l=1}^d\), defined at every point \(\mathbf{x} \in \mathbb{R}^d\):

\begin{enumerate}
\item Symmetry: Each component \(C_{ijkl} \in C^2(\overline{\Omega})\), and the tensor satisfies the classical symmetry conditions:
\begin{equation*}
C_{ijkl} = C_{jikl}, \quad C_{ijkl} = C_{ijlk}, \quad C_{ijkl} = C_{klij}, \quad \forall\, 1 \leq i,j,k,l \leq d.
\end{equation*}

\item Coercivity: For every symmetric matrix \(\bm{\xi} = [\xi_{ij}] \in \mathbb{R}^{d \times d}_{\mathrm{sym}}\) and almost every \(\mathbf{x} \in \Omega\), the tensor satisfies the uniform ellipticity condition:
\[
\sum_{i,j,k,l=1}^d C_{ijkl}(\mathbf{x}) \, \xi_{ij} \, \xi_{kl} \geq \Lambda |\bm{\xi}|^2,
\]
for some constant \(\Lambda > 0\).
\end{enumerate}

The inverse problem formulated in Problem~\ref{isp} is, in general, ill-posed, and uniqueness of the solution cannot be guaranteed without further assumptions. Specifically, recovering both the initial displacement $\mathbf{p}$ and the initial velocity $\mathbf{q}$ from boundary measurements may admit multiple solutions. If either $\mathbf{p}$ or $\mathbf{q}$ is prescribed, then the other might be uniquely determined. 
In this work, we aim to reconstruct both $\mathbf{p}$ and $\mathbf{q}$, and thus cannot assume prior knowledge of either. To resolve the issue of non-uniqueness, we adopt a Moore-Penrose-type approach: among all admissible pairs $(\mathbf{p}, \mathbf{q})$ consistent with the observed boundary data, we select the one for which the corresponding displacement field $\mathbf{u}$ attains some minimal $L^2$-norm.  One of the central theoretical results of this paper is a rigorous proof that this Moore-Penrose-type selection yields a unique solution to the inverse problem.

 The inverse problem of recovering the initial conditions $(\mathbf{p}, \mathbf{q})$ in the elastic wave equation has direct relevance in a range of real-world applications involving wave propagation in elastic media. One prominent example is seismology, where determining the initial displacement and velocity of seismic waves from surface measurements is crucial for understanding earthquake dynamics and imaging the Earth's interior~\cite{aki2002quantitative, larmat2006time, norville2013time, virieux2009overview}. In such settings, neither the initial displacement nor the velocity field is fully known, and reconstructing them from boundary sensors can provide critical insights into subsurface structures and fault zones.
Another important domain is nondestructive testing, where elastic waves are used to detect internal defects in solid materials. By measuring wave responses on the surface of an object, one can infer internal initial disturbances caused by cracks or inclusions~\cite{rose2014ultrasonic}. Similarly, in medical elastography, elastic waves are used to probe tissue stiffness and diagnose abnormalities such as tumors or fibrosis. Accurate reconstruction of the initial mechanical response within tissue, based on measurements from external transducers, is essential for producing high-resolution diagnostic images~\cite{alves2009full, fink2001acoustic,  green2019ultrasound, sarvazyan1998elasticity}.
These applications often involve anisotropic and inhomogeneous materials, such as layered geological media or biological tissues, which motivates the study of the elastic wave equation in such general settings. Moreover, in practice, only boundary data are available, and the internal state must be inferred indirectly. The ability to recover initial displacement and velocity fields from such data, particularly under minimal a priori assumptions, is thus of considerable practical interest.

Solving the inverse source problem, and in particular the inverse initial data problem, for the elasticity equation is challenging due to the coexistence of compressional and shear waves.
Several methods have been proposed in the literature to address inverse initial data problems for wave-type equations, including the elastic wave equation. A fundamental theoretical tool is the use of Carleman estimates, which have been successfully applied to establish uniqueness and stability results when partial initial data is known. These estimates provide strong conditional uniqueness guarantees and are thoroughly treated in the literature~\cite{isakov2006inverse, yamamoto2009carleman}. For numerical reconstruction, time-reversal methods offer an elegant approach by exploiting the time-symmetry of wave equations. Originally developed in the context of acoustics~\cite{fink1997time}, they have since been adapted to elastodynamics for reconstructing sources and initial states, particularly when full boundary data is available.
Another class of approaches relies on optimization-based formulations, where the initial data are obtained by minimizing a misfit between observed and simulated boundary data, often under regularization. These methods are flexible and effective in practical scenarios, especially when adjoint-state techniques are used to compute gradients efficiently~\cite{engl1996regularization}. In cases where the forward problem is linear and underdetermined, the Moore-Penrose pseudoinverse framework provides a principled way to select a unique solution by enforcing a minimal-norm criterion. This idea underlies many regularization schemes and is particularly well-suited for inverse problems with non-unique solutions~\cite{hansen2010discrete}. For simple geometries, spectral decomposition methods can also be applied, leveraging eigenfunction expansions to reconstruct the initial data from boundary information, though their applicability is limited to highly idealized settings.
Together, these approaches form a diverse toolkit, each with specific strengths: Carleman estimates provide theoretical uniqueness, time-reversal offers intuitive and direct numerical recovery, optimization and regularization frameworks accommodate noisy data, and pseudoinverse-based methods yield canonical solutions in ill-posed scenarios.

On the other hand, we refer the reader to the topic review~\cite{bonnet2005inverse}, the book~\cite{ammari2015mathematical}, and the references therein for the inverse source problem in the time-harmonic regime, as well as to~\cite{ammari2013time,bao2018inverse} for studies addressing the time-domain formulation, where source terms appear explicitly in the elastic equation rather than in the initial conditions. 
In~\cite{ammari2013time}, a time-reversal technique was employed to reconstruct the source function. A crucial requirement of this approach is access to the wave field at the final time $ t = T$ throughout the domain $ \Omega, $ which can be estimated when boundary reflections are negligible. However, in our setting, the time-reversal method is inapplicable since estimating the wave field at the final time is challenging, especially due to the lack of both initial data.
The work in~\cite{bao2018inverse} also addresses inverse source problems for elasticity in both frequency and time domains. Their approach relies on explicit Green's functions, which restricts applicability to homogeneous and unbounded media. In contrast, the present study considers a fully inhomogeneous and anisotropic elastic medium without imposing structural assumptions on the elasticity tensor \( \mathbb{C} \), such as isotropy or scalar-matrix form. To the best of our knowledge, this general inverse problem has not been previously investigated.

We now discuss additional aspects of the problem.
Inverse problems of identifying the body force for the elastic wave equation have been studied under various physical and mathematical settings. In the case of a homogeneous isotropic medium with constant Lam\'e parameters, several works have investigated the recovery of the body force term in the Lam\'e system using boundary observations. For instance, early-time measurements have been used to reconstruct moment tensors and source locations in gravitational models \cite{baldassari2024early}. Time-separable body forces have been addressed using Fourier-based regularization methods in bounded domains \cite{trong2009determination}. In full-space configurations, time-derivative Dirac impulses as sources have been reconstructed via weighted time-reversal strategies \cite{ammari2011time}. Extensions of time-reversal methods to elastic systems with separable source forms have also been explored \cite{brevis2019source}. 
In the context of thermoacoustic tomography, the problem of recovering initial displacement from boundary measurements has attracted attention. Under appropriate assumptions on the Lam\'e parameters, uniqueness and reconstruction strategies have been established \cite{tittelfitz2012thermoacoustic}, and convergence of the time-reversal-based Neumann series approach has been demonstrated without restrictive assumptions on wave speeds \cite{katsnelson2018convergence}.
When the elastic properties of the medium are neglected, the inverse source problem can be modeled using scalar hyperbolic equations. This scalar version of the inverse source problem has been the subject of extensive research in the literature; see, for example, \cite{Acosta:jde2018, Ammarietal:cm2011,Ammarielal:sp2012, Burgholzer:pspie2007,DoKunyansky:ip2018,fink2001acoustic,Duvaut1976,Haltmeier:cma2013,Kowar:SISI2014, LeNguyenNguyenPowell:JOSC2021,Natterer:ipi2012,NguyenKlibanov:ip2022, Linh:ipi2009}.

In this work, we introduce a novel orthonormal basis for the time variable, constructed by combining Legendre polynomials with an exponential weight. This Legendre polynomial--exponential basis enjoys key approximation properties such as orthogonality, smoothness, and non-degeneracy. These structural features enable the stable representation of time-dependent functions and spectrally facilitate differentiation of time components. Compared to previous bases (e.g., the polynomial--exponential basis in \cite{Klibanov:jiip2017}), the Legendre--exponential basis provides enhanced numerical stability and approximation quality, especially near the initial time.

Building on this basis, we propose a time-dimensional reduction method to solve the inverse problem of reconstructing initial displacement and velocity fields in elastic wave equations. The idea is to express the space-time solution as a finite series in the time basis, reducing the inverse problem to a sequence of elliptic PDEs in the spatial domain. This approach significantly reduces computational complexity and transforms the ill-posed inverse problem into a well-structured optimization problem over Fourier coefficients.

To justify the method, we prove a convergence theorem which ensures that, under suitable choices of regularization parameter and truncation level, the reduced solution converges strongly to the true solution in appropriate Sobolev spaces. This theoretical result guarantees that the reconstruction is reliable, even in the presence of noise, and provides a rigorous foundation for the proposed numerical scheme.

The remainder of the paper is organized as follows. Section~\ref{sec:time_basis} introduces the Legendre polynomial-exponential basis and discusses its crucial structural properties that motivate its use in the proposed approach. Section~\ref{sec_timereduce} presents the time-dimensional reduction method based on this basis. Section~\ref{sec_convergence} is devoted to establishing convergence results and stability estimates for the regularized solution. Section~\ref{sec_num} contains several numerical experiments that demonstrate the accuracy and robustness of the method under different reconstruction scenarios. Finally, Section~\ref{sec_concluding} offers concluding remarks.

\section{The Legendre polynomial-exponential basis}\label{sec:time_basis}

In this section, we introduce the Legendre polynomial-exponential basis and highlight several of its key properties that are instrumental in establishing the convergence of the time-dimensional reduction method. This basis is formed by multiplying the classical Legendre polynomials by an exponential function. Compared to the polynomial-exponential basis introduced in~\cite{Klibanov:jiip2017}, the Legendre polynomial-exponential basis possesses structural advantages that make it more suitable for the convergence of the time reduction model.

Let $n \geq 0$ be fixed. The Legendre polynomial of degree $n$ on the interval $(-1, 1)$ is defined by Rodrigues' formula
\begin{equation}
P_n(x) = \frac{1}{2^n n!} \frac{{\rm d}^n}{{\rm d}x^n} \left( x^2 - 1 \right)^n.
\label{Rodrigues}
\end{equation}
To construct a corresponding orthonormal family on the finite interval $(0, T)$, the change of variables $x = \frac{2t}{T} - 1$ is applied, mapping $(0, T)$ onto $(-1, 1)$. This yields the transformed functions
\begin{equation}
Q_n(t) =  
\sqrt{\frac{2n + 1}{T}} P_n\left(\frac{2t}{T} - 1\right).
\label{Q_n_def}
\end{equation}
The set $\{ Q_n \}_{n \geq 0}$ forms an orthonormal basis for the Hilbert space $L^2(0, T)$.
Since the Legendre polynomials satisfy the self-adjoint equation
\begin{equation}
\left[ (1 - x^2) P_n'(x) \right]' + n(n+1) P_n(x) = 0, \quad x \in (-1, 1),
\label{seft_adjoint_P}
\end{equation}
by using the product rule with the change of variable $t \mapsto x = 2t/T - 1$, we can derive from \eqref{seft_adjoint_P} the analogous self-adjoint differential equation for $Q_n$, read as
\begin{equation}
\left[ t(T - t) Q_n'(t) \right]' + n(n + 1) Q_n(t) = 0, \quad t \in (0, T).
\label{2.11111}
\end{equation}

Defining
\begin{equation}
\Psi_n(t) = e^t Q_n(t),
\quad
t \in (0, T), n \in \mathbb{N}.
\label{2.1}
\end{equation}
The set of functions $\{\Psi_n\}_{n \geq 0}$ form an orthonormal basis of the weighted Hilbert space
\[
L^2_{ e^{-2t}}(0, T) = \left\{ u : \int_0^T e^{-2t} |u(t)|^2 \, {\rm d}t < \infty \right\},
\]
equipped with the inner product
\[
\langle u, v \rangle_{L_{e^{-2t}}^2(0, T)} = \int_0^T e^{-2t} u(t) v(t) \, {\rm d}t.  
\]

\begin{Definition}[Legendre-polynomial exponential basis]
The set $\{\Psi_n\}_{n \geq 0}$, defined in \eqref{2.1}, is referred to as the orthonormal Legendre-exponential basis of the space $L^2_{e^{-2t}}(0,T)$. 
\end{Definition}

The introduction of this basis represents one of the central conceptual contributions of the paper. It provides the foundational framework for constructing the time-dimensional reduction method and offers structural advantages essential for establishing convergence estimates and understanding the behavior of the time-dimensional reduction model, which approximates the original inverse problem. See Section \ref{sec_timereduce} for details. We introduce several concepts associated with the Legendre-polynomial exponential basis, which play a central role in our time dimension reduction framework.

\begin{Definition}

\begin{enumerate}
    \item For any $u \in L^2((0, T); L^2(\Omega))$, the $n$-th Fourier mode of $u$ with respect to the Legendre-polynomial exponential basis is defined as
    \begin{equation}
        u_n(\mathbf{x}) = \left\langle u(\mathbf{x}, \cdot), \Psi_n \right\rangle_{L^2_{e^{-2t}}(0, T)} = \int_0^T e^{-2t} u(\mathbf{x}, t) \Psi_n(t) \, {\rm d}t
        \label{Fourier_mode}
    \end{equation}
for $\x \in \Omega$.
    \item Given an integer $N \geq 0$ and a vector $U = \begin{bmatrix} u_0 & u_1 & \dots & u_N \end{bmatrix}^\top \in L^2(\Omega)^{N + 1}$, we define the space-time reconstruction (or expansion) of $U$ as
    \begin{equation}
        \mathbb{S}^N[U](\mathbf{x}, t) = \sum_{n = 0}^N u_n(\mathbf{x}) \Psi_n(t),
        \label{space_time_expansion}
    \end{equation}
    for all $(\mathbf{x}, t) \in \Omega_T:= \Omega \times (0, T)$.

    \item For $N \geq 0$, we define the finite-dimensional subspace
    \begin{equation}
        \mathbb{V}_N = \operatorname{span} \left\{ \Psi_0, \Psi_1, \dots, \Psi_N \right\} \subset L^2_{e^{-2t}}(0, T).
        \label{VN}
    \end{equation}
    Given a function $u \in L^2((0, T); L^2(\Omega))$, its orthogonal projection onto $\mathbb{V}_N$ in the time variable is given by
    \begin{equation}
        P^N u(\mathbf{x}, t) = \sum_{n = 0}^N \left\langle u(\mathbf{x}, \cdot), \Psi_n \right\rangle_{L^2_{e^{-2t}}(0, T)} \Psi_n(t),
        \label{PN}
    \end{equation}
    for all $(\mathbf{x}, t) \in \Omega_T$.
\end{enumerate}
\label{def2.2}
\end{Definition}

\begin{Proposition}
 For every $n \geq 0$, the function $\Psi_n(t) = e^t Q_n(t)$ is infinitely differentiable on the interval $(0, T)$, and none of its derivatives of any order is identically zero on this interval.
\label{prop2.1}
\end{Proposition}

\begin{proof}
Since $\Psi_n(t) = e^t Q_n(t)$ is the product of the analytic function $e^t$ and the polynomial $Q_n(t)$ of degree $n$, it follows that $\Psi_n$ is infinitely differentiable on $(0, T)$.
To show that $\Psi_n^{(k)}$ is not identically zero for any $k \geq 0$, we apply the Leibniz rule:
\begin{equation}
\Psi_n^{(k)}(t) = \frac{{\rm d}^k}{{\rm d}t^k}(e^t Q_n(t)) = \sum_{j=0}^k \binom{k}{j} e^t Q_n^{(j)}(t) = e^t \left( Q_n(t) + \sum_{j=1}^k \binom{k}{j} Q_n^{(j)}(t) \right).
\label{2.3}
\end{equation}
Here, $Q_n(t)$ is a degree-$n$ polynomial, while each derivative $Q_n^{(j)}(t)$ for $j \geq 1$ is a polynomial of degree strictly less than $n$. Thus, the expression inside the parentheses in \eqref{2.3} remains a degree-$n$ polynomial.

Suppose, for contradiction, that $\Psi_n^{(k)}(t) \equiv 0$ on $(0, T)$. Then the polynomial
\[
Q_n(t) + \sum_{j=1}^k \binom{k}{j} Q_n^{(j)}(t)
\]
must vanish identically. However, since it is a polynomial of degree $n$, this would imply it has more than $n$ roots, which is only possible if it is the zero polynomial. This is a contradiction, and hence, $\Psi_n^{(k)}(t)$ cannot vanish identically.
\end{proof}

\begin{Lemma}
    For every \( k \in \mathbb{N} \), there exists a constant \( C_1 > 0 \), depending only on \( k \) and \( T \), such that
    \begin{equation} \label{Hk-inequality}
\sum_{n = 0}^\infty n^{2k} \left| \langle u, \Psi_n \rangle_{L^2_{e^{-2t}}(0, T)} \right|^2 \leq C \|u\|_{H^k(0, T)}^2,
\end{equation}
for all $u \in H^k(0, T)$ where $C$ is a constant depending only on $T.$
\label{lem2.1}
\end{Lemma}

\begin{proof}
    We define the differential operator \( K : H^2(0, T) \to L^2(0, T) \) by
\[
Ku(t) := -\frac{{\rm d}}{{\rm d}t} \left( t(T - t) \frac{{\rm d}u}{{\rm d}t} \right), \quad \text{for } u \in H^2(0, T),\ t \in (0, T).
\]
From the self-adjoint differential equation \eqref{2.11111}, it follows that
\[
K Q_n = n(n + 1) Q_n.
\]
Using the eigenvalue relation we have \[ K^m Q_n = n^m(n + 1)^m Q_n.\]
 Hence, we obtain:
\begin{align*}
	\langle K^m(e^{-t} u), Q_n \rangle_{L^2(0, T)} 
	&= \langle e^{-t} u, K^m Q_n \rangle_{L^2(0, T)} \\
	&= n^m(n+1)^m \langle e^{-t} u, Q_n \rangle_{L^2(0, T)} \\
	&= n^m(n+1)^m \int_0^T e^{-2t} u(t) \Psi_n(t) \, {\rm d}t \\
	&= n^m(n + 1)^m \langle u, \Psi_n \rangle_{L^2_{e^{-2t}}(0, T)} = n^m(n + 1)^m u_n,
\end{align*}
where \( u_n := \langle u, \Psi_n \rangle_{L^2_{e^{-2t}}(0, T)} \). 

We consider two cases:

Case 1: \( k = 2m \) for some integer \( m \geq 1 \).

Applying Parseval's identity in \( L^2(0, T) \) for the orthonormal basis \( \{ Q_n \} \), we obtain:
\[
\sum_{n = 0}^\infty n^{2m}(n + 1)^{2m} |u_n|^2 
= \sum_{n = 0}^\infty \left| \langle K^m(e^{-t} u), Q_n \rangle_{L^2(0, T)} \right|^2 
= \| K^m(e^{-t} u) \|_{L^2(0, T)}^2.
\]
Since multiplication by \( e^{-t} \) is a bounded operator on \( H^{2m}(0, T) \), we deduce that
\[
\| K^m(e^{-t} u) \|_{L^2(0, T)} \leq C \| u \|_{H^{2m}(0, T)},
\]
which implies \eqref{Hk-inequality} in this case.

Case 2: \( k = 2m + 1 \).
In this case, we use the spectral energy identity for \( \{ Q_n \} \), which states:
\[
\int_0^T t(T - t) \left| \left( K^m(e^{-t} u) \right)' \right|^2 {\rm d}t = \sum_{n = 1}^\infty n(n + 1) \left| \langle K^m(e^{-t} u), Q_n \rangle_{L^2(0, T)} \right|^2.
\]
Substituting the expression for the Fourier coefficients, we have:
\[
\langle K^m(e^{-t} u), Q_n \rangle_{L^2(0, T)} = n^m(n + 1)^m u_n,
\]
and therefore,
\[
\int_0^T t(T - t) \left| \left( K^m(e^{-t} u) \right)' \right|^2 {\rm d}t = \sum_{n = 1}^\infty n^{2m+1}(n + 1)^{2m+1} |u_n|^2,
\]
which establishes the desired estimate \eqref{Hk-inequality} for odd \( k \) as well.
\end{proof}

\begin{Lemma}
    For all $n \geq 1$,
    \begin{equation*}
        \|\Psi_n'\|_{L^2_{e^{-2t}}(0, T)} \leq C n^{3/2},
        \quad
        \mbox{and}
        \quad
        \|\Psi_n''\|_{L^2_{e^{-2t}}(0, T)} \leq C n^{7/2}
    \end{equation*}
    where $C$ is a constant depending only on $T$.
    \label{lem2.2}
\end{Lemma}
\begin{proof}
 For convenience,  throughout this proof, we use $C$ to denote a generic positive constant that may vary from line to line but depends only on $T$.

By differentiating Rodrigues' formula \eqref{Rodrigues} and using direct algebra, we obtain
\begin{equation}
P_n'(x) = (2n - 1) P_{n-1}(x) + P_{n-2}'(x), \quad \text{for all } n \geq 2.
\label{2.8}
\end{equation}
By recursively applying the identity \eqref{2.8}, we derive the expansion
\begin{align}
P_n'(x)
&= (2n - 1) P_{n-1}(x) + P_{n-2}'(x)  \notag\\
&= (2n - 1) P_{n-1}(x) + (2n - 5) P_{n-3}(x) + P_{n-4}'(x) \notag\\
&= \sum_{k = 0}^{\lfloor n/2 \rfloor - 1} \left(2n - (4k + 1)\right) P_{n - (2k + 1)}(x) + P_r'(x),
\label{2.1313}
\end{align}
where
\[
r =
\begin{cases}
0, & \text{if } n \text{ is even}, \\
1, & \text{if } n \text{ is odd},
\end{cases}
\quad \mbox{and }
\quad
P'_r(x) =
\begin{cases}
0, & \text{if } n \text{ is even}, \\
1, & \text{if } n \text{ is odd}.
\end{cases}
\]
It follows from \eqref{2.1313} that
\begin{equation}
	\left\|P_n' - P_r'\right\|^2_{L^2(-1, 1)} = 
	\left\| \sum_{k = 0}^{\lfloor n/2 \rfloor - 1} \left(2n - (4k + 1)\right) P_{n - (2k + 1)} \right\|^2_{L^2(-1, 1)}.
\end{equation}

Since the Legendre polynomials $\{P_n\}_{n \geq 0}$ form an orthogonal basis in $L^2(-1, 1)$, and satisfy the norm identity
\[
\|P_n\|_{L^2(-1, 1)}^2 = \frac{2}{2n + 1},
\]
we can compute the squared norm of the truncated expansion:
\begin{align}
\|P_n' -  P_r'\|_{L^2(-1, 1)}^2
&= \sum_{k = 0}^{\lfloor n/2 \rfloor - 1} \left(2n - (4k + 1)\right)^2 \cdot \|P_{n - (2k + 1)}\|_{L^2(-1, 1)}^2 \notag\\
&= \sum_{k = 0}^{\lfloor n/2 \rfloor - 1} \left(2n - (4k + 1)\right)^2 \cdot \frac{2}{2n - (4k + 1)} \notag \\
&= 2 \sum_{k = 0}^{\lfloor n/2 \rfloor - 1} \left(2n - (4k + 1)\right) \leq C n^2,
\label{2.14}
\end{align}
for some constant $C > 0$ independent of $n$.
Since $P'_r$ is either $0$ or $1$, it follows from \eqref{2.14} that
 \begin{equation*}
 	\|P_n'\|_{L^2(-1, 1)}^2 \leq C n^2.
 \end{equation*}
Using the chain rule and the definition of \( Q_n \) in \eqref{Q_n_def}, we obtain the analog estimate
\begin{equation}
\|Q_n'\|_{L^2(0, T)}  \leq Cn^{3/2}.
\label{2.9}
\end{equation}
We next estimate \( \|\Psi_n'\|_{L^2_{e^{-2t}}(0, T)} \). From \eqref{2.1} and \eqref{2.9}, we have
\begin{align*}
    \|\Psi_n'\|_{L^2_{e^{-2t}}(0, T)}^2 
    &= \int_0^T e^{-2t} |\Psi_n'(t)|^2 {\rm d}t 
    = \int_0^T e^{-2t} |e^t Q_n'(t) + e^t Q_n(t)|^2 {\rm d}t \\
    &= \int_0^T |Q_n'(t) + Q_n(t)|^2 {\rm d}t
    \leq 2 \|Q_n'\|_{L^2(0,T)}^2 + 2 \|Q_n\|_{L^2(0,T)}^2 
    \leq C n^3.
\end{align*}
Therefore,
\[
\|\Psi_n'\|_{L^2_{e^{-2t}}(0, T)} \leq C n^{3/2}.
\]

To estimate $\|\Psi_n''\|_{L^2_{e^{-2t}}(0, T)}$, we begin by estimating the second derivative of the Legendre polynomial $P_n$. 
From  the identity \eqref{2.1313} and the fact that $ P_r''(x) = 0$, differentiation yields \[
P_n''(x) = \sum_{k=0}^{\left\lfloor n/2 \right\rfloor - 1} (2n - (4k + 1)) P_{n - (2k + 1)}'(x).
\]

Using the triangle inequality and the known estimate $\|P_k'\|_{L^2(-1, 1)} \leq C k$, we deduce
\[
\|P_n''\|_{L^2(-1, 1)}\leq \sum_{k = 0}^{\left\lfloor n/2 \right\rfloor-1} (2n - (4k + 1)) \cdot \|P_{n - (2k + 1)}'\|_{L^2(-1, 1)} \leq C \sum_{k=0}^{\left\lfloor\frac{n-1}{2}\right\rfloor} n (n - 2k) \leq C n^3.
\]

We now estimate $\|Q_n''\|_{L^2(0, T)}$ based on the definition \eqref{Q_n_def}.
Differentiating twice and applying the chain rule gives
\[
Q_n''(t) = \frac{4 \sqrt{2n + 1}}{T^{5/2}} \cdot P_n''\left( \frac{2t}{T} - 1 \right).
\]
By the change of variable $x = \frac{2t}{T} - 1$, it follows that
\[
\|Q_n''\|_{L^2(0, T)}^2 = \frac{4(2n + 1)}{T^3} \cdot \|P_n''\|_{L^2(-1, 1)}^2 \leq C n^7.
\]

Finally, to estimate $\|\Psi_n''\|_{L^2_{e^{-2t}}(0, T)}$, we use the identity
\[
\Psi_n(t) = e^t Q_n(t) \quad \Rightarrow \quad \Psi_n''(t) = e^t(Q_n''(t) + 2 Q_n'(t) + Q_n(t)).
\]
Thus,
\[
\|\Psi_n''\|_{L^2_{e^{-2t}}(0, T)}^2 = \int_0^T e^{-2t} |\Psi_n''(t)|^2 {\rm d}t = \int_0^T |Q_n''(t) + 2 Q_n'(t) + Q_n(t)|^2 {\rm d}t.
\]
Applying the inequality $(a + b + c)^2 \leq 3(a^2 + b^2 + c^2)$ and using the estimates $\|Q_n\|_{L^2(0, T)} = 1$ and $\|Q_n'\|_{L^2(0, T)} \leq C n^{3/2}$, we obtain
\[
\|\Psi_n''\|_{L^2_{e^{-2t}}(0, T)}^2 \leq 3\left( \|Q_n''\|_{L^2(0, T)}^2 + 4\|Q_n'\|_{L^2(0, T)}^2 + \|Q_n\|_{L^2(0, T)}^2 \right) \leq C(n^7 + n^3 + 1) \leq C n^7.
\]
Hence, $\|\Psi_n''\|_{L^2_{e^{-2t}}(0, T)} \leq C n^{\frac{7}{2}}$.

\end{proof}

The following property plays the key role in the study of the behavior of the time dimensional reduction model in Section \ref{sec_timereduce}

\begin{Theorem} Let $p\geq 0$ and $u\in L^2((0,T); H^p(\Omega))$.
If  the series
\begin{equation}
\sum_{n = 0}^\infty \langle u(\cdot, \cdot), \Psi_n \rangle_{L^2_{e^{-2t}}(0, T)} \, \Psi_n''(t)
\label{utt}	
\end{equation}
converges in $ L^2((0,T); H^p(\Omega))$, then the generalized derivatives $u_{tt}$ exists in 
$ L^2((0,T); H^p(\Omega))$ and the series coincides with $u_{tt}$.

(b) If $u \in H^k((0, T); H^p(\Omega))$ with $k \geq 5$ then the series \eqref{utt} converges in $ 
L^2((0,T); H^p(\Omega))$.
\label{thm2.1}
\end{Theorem}

\begin{proof}
(a) Define the limit function
\[
w(\mathbf{x}, t) := \lim_{N \to \infty} \sum_{n = 0}^N \langle u(\cdot, \cdot), \Psi_n \rangle_{L^2_{e^{-2t}}(0, T)} \, \Psi_n''(t).
\]
We  show that $w = u_{tt}$. Let $\varphi \in C_c^\infty(\Omega_T)$ be an arbitrary test function. We compute:
\begin{align*}
	\int_{\Omega_T} w(\mathbf{x}, t)\varphi(\mathbf{x}, t)\, {\rm d}\mathbf{x}{\rm d}t 
	&= \int_{\Omega_T} \lim_{N \to \infty}  \left( \sum_{n = 0}^N \langle u(\x, \cdot), \Psi_n \rangle_{L^2_{e^{-2t}}(0, T)} \Psi_n''(t) \right) \varphi(\mathbf{x}, t)\, {\rm d}\mathbf{x}{\rm d}t \\
	&= \lim_{N \to \infty} \sum_{n = 0}^N \int_{\Omega_T} \langle u(\x, \cdot), \Psi_n \rangle_{L^2_{e^{-2t}}(0, T)} \, \Psi_n''(t)\, \varphi(\mathbf{x}, t)\, {\rm d}\mathbf{x}{\rm d}t.
\end{align*}
This interchange of the limit and the integral above is justified by  the convergence in $L^2((0,T); H^p(\Omega))$ of the partial sums to $w$, while $\varphi$ is compactly supported and bounded.
Next, we integrate by parts in time:
\begin{align*}
	\lim_{N \to \infty} \sum_{n = 0}^N \int_{\Omega_T} \langle u(\x, \cdot), \Psi_n &\rangle_{L^2_{e^{-2t}}(0, T)} \Psi_n''(t) \varphi(\mathbf{x}, t)\, {\rm d}\mathbf{x}{\rm d}t 
	\\
	&= \lim_{N \to \infty} \sum_{n = 0}^N \int_\Omega \langle u(\mathbf{x}, \cdot), \Psi_n \rangle_{L^2_{e^{-2t}}(0, T)} \left( \int_0^T \Psi_n''(t) \varphi(\mathbf{x}, t) {\rm d}t \right) {\rm d}\mathbf{x} \\
	&=  \int_\Omega \lim_{N \to \infty} \sum_{n = 0}^N  \langle u(\mathbf{x}, \cdot), \Psi_n \rangle_{L^2_{e^{-2t}}(0, T)} \left( \int_0^T \Psi_n(t) \varphi_{tt}(\mathbf{x}, t) {\rm d}t \right) {\rm d}\mathbf{x} \\
	&= \int_{\Omega_T} u(\mathbf{x}, t) \varphi_{tt}(\mathbf{x}, t)\,{\rm d}\mathbf{x}{\rm d}t,
\end{align*}
where we used completeness of the $\{\Psi_n\}$ basis in $L^2_{e^{-2t}}(0, T)$,  convergence of the series, and the dominated convergence theorem again.
Hence, $w = u_{tt}$ in the distributional sense, and since both sides lie in $L^2((0, T); L^2(\Omega))$, the equality holds almost everywhere.

(b)
We establish the convergence of the series in $L^2((0, T); H^p(\Omega))$. From Lemmas~\ref{lem2.1} and~\ref{lem2.2}, we have the estimates
\[
\left\| \langle u(\cdot, \cdot), \Psi_n \rangle_{L^2_{e^{-2t}}(0, T)} \right\|_{H^p(\Omega)} \leq C n^{-k},
\qquad 
\| \Psi_n'' \|_{L^2_{e^{-2t}}(0, T)} \leq C n^{7/2}.
\]
Therefore,
\begin{align*}
\sum_{n = 0}^\infty \left\| \langle u(\cdot, \cdot), \Psi_n \rangle_{L^2_{e^{-2t}}(0, T)} \Psi_n''(t) \right\|_{L^2((0, T); H^p(\Omega))}
&\leq \sum_{n = 0}^\infty \left\| \langle u(\cdot, \cdot), \Psi_n \rangle_{L^2_{e^{-2t}}(0, T)} \right\|_{H^p(\Omega)} \cdot \| \Psi_n'' \|_{L^2_{e^{-2t}}(0, T)} \\
&\leq C \sum_{n = 0}^\infty n^{-k} \cdot n^{7/2} = C \sum_{n = 0}^\infty n^{7/2 - k}.
\end{align*}
The series on the right converges whenever $k \geq  5$, so the sum is convergent in $L^2((0, T); H^p(\Omega))$.
This completes the proof.
\end{proof}

\begin{remark}
The convergence result established in Theorem~\ref{thm2.1} crucially relies on the special structure and analytical properties of the Legendre polynomials $\{P_n\}$. In particular, each polynomial $P_n$ is an eigenfunction of the Sturm--Liouville operator \[ u \mapsto -\frac{{\rm d}}{{\rm d}x} \left[ (1 - x^2) \frac{{\rm d}u}{{\rm d}x} \right], \]
which endows them with strong regularity, decay properties, and orthogonality in weighted Sobolev spaces. These features directly lead to the decay estimates for the Fourier coefficients and controlled growth of their second derivatives, which are key in proving the convergence of the second-order expansion in time.

Without leveraging the spectral theory of Legendre polynomials and the corresponding eigenvalue structure, establishing convergence of the time-dimensional reduction method becomes significantly more difficult. Indeed, in our earlier works, see e.g., \cite{ HaoThuyLoc:2024, LeNguyenNguyenPark,  NguyenNguyenVu, DangNguyenVu}, the lack of such a framework posed major analytical challenges in justifying the convergence of analogous series representations.
\end{remark}

\section{The uniqueness of minimum solution and the time-dimensional reduction method} \label{sec_timereduce}

To numerically solve Problem~\ref{isp}, we aim to approximate the exact displacement field $\mathbf{u}$, which satisfies the following boundary value problem:
\begin{equation}
\begin{cases}
\displaystyle \frac{\partial^2 \mathbf{u}(\mathbf{x}, t)}{\partial t^2} = \Div \left( \mathbb{C} : \nabla \mathbf{u}(\mathbf{x}, t) \right), & \text{in } \Omega_T , \\
\mathbf{u}(\mathbf{x}, t) = \mathbf{f}^*(\mathbf{x}, t), & \text{on } \Gamma_T := \partial \Omega \times (0, T), \\
\partial_{\nu} \mathbf{u}(\mathbf{x}, t) = \mathbf{g}^*(\mathbf{x}, t), & \text{on } \Gamma_T,
\end{cases}
\label{2.1111}
\end{equation}
where $\mathbf{f}^*$ and $\mathbf{g}^*$ denote the exact boundary displacement and traction data. In practice, only noisy measurements $\mathbf{f}$ and $\mathbf{g}$ are available for computation. We will discuss the noise analysis later in Section \ref{sec_convergence}. In this section, we study the minimum-norm solution and the time-dimensional reduction method to compute it.

It is well known that problem~\eqref{2.1111} admits at least one solution, as it corresponds to the restriction of the solution to the initial value problem~\eqref{main_ivp} on the domain $\Omega$. However, since no initial conditions are specified, the solution is generally nonunique.
Among all possible solutions to \eqref{2.1111}, we choose to compute the one with the smallest norm in a suitable Hilbert space. In fact, we use the Hilbert space
$L^2_{e^{-2t}}((0,T); H^{2+p}(\Omega)^d)$ which is $L^2((0,T); H^{2+p}(\Omega)^d)$ with the inner product
$$ \langle\mathbf{v},\mathbf{w}\rangle_{L^2_{e^{-2t}}((0,T); H^{2+p}(\Omega)^d)}=\int_0^Te^{-2t}
\langle \mathbf{v}(t),\mathbf{w}(t)\rangle_{H^{2+p}(\Omega^d)}{\rm d}t.  $$
This minimum norm solution, motivated by the Moore-Penrose pseudoinverse framework, is unique and provides a stable and well-defined approximation. The use of minimal norm solutions is widely accepted in the inverse problems community, particularly when dealing with ill-posed or underdetermined systems.

\begin{Theorem}[Existence and Uniqueness of a Minimal-Norm Solution]
Let $p \geq 0$ and define the set of admissible solutions
\[
S = \left\{ \mathbf{u} \in L^2\left(0,T; (H^{2 + p}(\Omega))^d\right) \cap H^2\left(0,T; (H^p(\Omega))^d\right) \,\middle|\, \mathbf{u} \text{ satisfies } \eqref{2.1111} \right\}.
\]
Assume that the system \eqref{2.1111} admits at least one solution in $S$. Then, there exists a unique function $\mathbf{u}^* \in S$ such that
\[
\mathbf{u}^* = \operatorname*{argmin}_{\mathbf{u} \in S} \left\| \mathbf{u} \right\|_{L^2_{e^{-2t}}\left(0,T; (H^{2 + p}(\Omega))^d\right)}.
\]
\label{min_solution}
\end{Theorem}

\begin{proof}
Define
\[
\boldsymbol{\alpha} = \inf\left\{ \|\mathbf{u}\|_{L^2_{e^{-2t}}(0,T; (H^{2 + p}(\Omega))^d)} : \mathbf{u} \in S \right\}.
\]
Let $\{\mathbf{u}^n\} \subset S$ be a minimizing sequence such that
\begin{equation}
\boldsymbol{\alpha} \leq \|\mathbf{u}^n\|_{L^2_{e^{-2t}}(0,T; (H^{2 + p}(\Omega))^d)} \leq \boldsymbol{\alpha} + \frac{1}{n}.
\label{to-min}
\end{equation}
Since $\mathbf{u}^n$ satisfies the wave equation $\mathbf{u}^n_{tt} = \operatorname{div}(\mathbb{C} : \nabla \mathbf{u}^n)$, and $\mathbb{C}$ is smooth, the second time derivatives $\mathbf{u}^n_{tt}$ are uniformly bounded in $L^2_{e^{-2t}}(0,T; (H^p(\Omega))^d)$. Therefore, there exists a subsequence (still denoted $\mathbf{u}^n$) and a function $\mathbf{u} \in S$ such that
\[
\mathbf{u}^n \rightharpoonup \mathbf{u} \text{ weakly in } L^2_{e^{-2t}}(0,T; (H^{2 + p}(\Omega))^d), \quad 
\mathbf{u}^n_{tt} \rightharpoonup \mathbf{u}_{tt} \text{ weakly in } L^2_{e^{-2t}}(0,T; (H^p(\Omega))^d).
\]
By weak lower semicontinuity and \eqref{to-min}, it follows that
\[
\|\mathbf{u}\|_{L^2_{e^{-2t}}(0,T; (H^{2 + p}(\Omega))^d)} = \boldsymbol{\alpha}.
\]
To show uniqueness, suppose there exists another $\mathbf{v} \in S$ such that 
\[
\|\mathbf{v}\|_{L^2_{e^{-2t}}(0,T; (H^{2 + p}(\Omega))^d)} = \boldsymbol{\alpha}.
\]
Then the average $\tfrac{1}{2}(\mathbf{u} + \mathbf{v})$ also belongs to $S$, and satisfies
\[
\left\| \tfrac{1}{2}(\mathbf{u} + \mathbf{v}) \right\|_{L^2_{e^{-2t}}(0,T; (H^{2 + p}(\Omega))^d)} \geq \boldsymbol{\alpha}.
\]
Applying the parallelogram identity gives
\begin{align*}
\|\mathbf{u} - \mathbf{v}\|^2_{L^2_{e^{-2t}}(0,T; (H^{2 + p}(\Omega))^d)} 
&= 2 \|\mathbf{u}\|^2_{L^2_{e^{-2t}}(0,T; (H^{2 + p}(\Omega))^d)} + 
2 \|\mathbf{v}\|_{L^2_{e^{-2t}}(0,T; (H^{2 + p}(\Omega))^d)}^2\\
& - 4 \left\| \tfrac{1}{2}(\mathbf{u} + \mathbf{v}) \right\|^2_{L^2_{e^{-2t}}(0,T; (H^{2 + p}(\Omega))^d)} \\
&= 4 \boldsymbol{\alpha}^2 - 4 \left\| \tfrac{1}{2}(\mathbf{u} + \mathbf{v}) \right\|^2_{L^2_{e^{-2t}}(0,T; (H^{2 + p}(\Omega))^d)} \leq 0.
\end{align*}
Hence, $\mathbf{u} = \mathbf{v}$, and the minimal-norm solution is unique.
\end{proof}

We now develop a numerical strategy to approximate the minimal-norm solution of the boundary value problem~\eqref{2.1111}. To this end, we employ a spectral representation in time using the new Legendre polynomial-exponential basis $\{\Psi_n\}_{n \geq 0}$ of the weighted space $L^2_{e^{-2t}}(0, T)$. For each spatial point $\mathbf{x} \in \Omega$, the displacement field $\mathbf{u}(\mathbf{x}, t)$ is approximated by its projection on to $\mathbb{V}_N$ in the time variable:
\begin{equation}
    \mathbf{u}(\mathbf{x}, t) = \sum_{n = 0}^{\infty} \mathbf{u}_n(\mathbf{x}) \Psi_n(t)
    \approx P^N \bu(\x, t) =  \sum_{n = 0}^{N} \mathbf{u}_n(\mathbf{x}) \Psi_n(t), \quad (\mathbf{x}, t) \in \Omega_T,
    \label{u_appr}
\end{equation}
where the coefficients $\mathbf{u}_n$ are given by the $n-th$ Fourier mode of $\bu$, see \eqref{Fourier_mode}
\begin{equation}
    \mathbf{u}_n(\mathbf{x}) = \langle \mathbf{u}(\mathbf{x}, \cdot), \Psi_n \rangle_{L^2_{e^{-2t}}(0, T)} = \int_0^T e^{-2t} \mathbf{u}(\mathbf{x}, t) \Psi_n(t) \, {\rm d}t.
    \label{u_coef}
\end{equation}
The truncation parameter $N$ will be chosen based on the desired accuracy and computational constraints.
Due to Theorem \ref{thm2.1},
\begin{equation}
    \partial_{tt} \mathbf{u}(\x, t) \approx \sum_{n = 0}^{N} \mathbf{u}_n(\mathbf{x}) \Psi_n''(t), \quad (\mathbf{x}, t) \in \Omega_T.
    \label{utt_appr}
\end{equation}
Substituting the approximations in ~\eqref{u_appr} and \eqref{utt_appr} into the elastic wave equation in~\eqref{2.1111}, we obtain:
\begin{equation}
    \sum_{n = 0}^{N} \mathbf{u}_n(\mathbf{x}) \Psi_n''(t)
    = \sum_{n = 0}^{N} \Div \left( \mathbb{C} : \nabla \mathbf{u}_n(\mathbf{x}) \Psi_n(t) \right),
    \quad (\mathbf{x}, t) \in \Omega_T.
    \label{eqn_appro}
\end{equation}
Multiplying both sides of~\eqref{eqn_appro} by $e^{-2t} \Psi_m(t)$ and integrating over $(0,T)$ yields the following elliptic system for each $m \in \{0, \dots, N\}$:
\begin{equation}
    \Div(\mathbb{C} : \nabla \mathbf{u}_m(\mathbf{x})) 
    - \sum_{n = 0}^N s_{mn} \mathbf{u}_n(\mathbf{x}) = 0,
    \label{time-reduction-model}
\end{equation}
where the matrix coefficients $s_{mn}$ are defined as
\[
    s_{mn} = \int_0^T e^{-2t} \Psi_n''(t) \Psi_m(t) \, {\rm d}t
    \quad 0 \leq m, n \leq N.
\]

The boundary conditions for the spatial coefficients $\mathbf{u}_m(\mathbf{x})$ are derived from the boundary data in \eqref{data_ip} and the Fourier coefficient formula \eqref{u_coef}, leading to:
\begin{align}
\mathbf{u}_m(\mathbf{x}) &= \int_0^T \mathbf{u}(\mathbf{x}, t) \Psi_m(t)e^{-2t} \, \,  {\rm d}t 
= \int_0^T \mathbf{f}(\mathbf{x}, t) \Psi_m(t) e^{-2t}\, \,  {\rm d}t 
=: \mathbf{f}_m(\mathbf{x}), \label{3.7} \\
\partial_\nu \mathbf{u}_m(\mathbf{x}) &= \int_0^T \partial_\nu \mathbf{u}(\mathbf{x}, t) \Psi_m(t)e^{-2t} \, \,  {\rm d}t 
= \int_0^T \mathbf{g}(\mathbf{x}, t) \Psi_m(t) e^{-2t}\, \,  {\rm d}t 
=: \mathbf{g}_m(\mathbf{x}), \label{3.8}
\end{align}
for all $\mathbf{x} \in \partial \Omega$ and $m = 0, \dots, N$.

Combining equations \eqref{time-reduction-model}, \eqref{3.7}, and \eqref{3.8}, we arrive at the time dimension reduction model for the problem we want to solve \eqref{2.1111}, namely
\begin{equation}
\begin{cases}
\ds\Div \left( \mathbb{C} : \nabla \mathbf{u}_m(\mathbf{x}) \right) 
- \sum_{n = 0}^{N} s_{mn} \mathbf{u}_n(\mathbf{x}) = 0, & \mathbf{x} \in \Omega, \\
\mathbf{u}_m(\mathbf{x}) = \mathbf{f}_m(\mathbf{x}), & \mathbf{x} \in \partial \Omega, \\
\partial_\nu \mathbf{u}_m(\mathbf{x}) = \mathbf{g}_m(\mathbf{x}), & \mathbf{x} \in \partial \Omega,
\end{cases}
\label{3.9}
\end{equation}
for $m = 0, 1, \dots, N$.

\begin{remark}
The system~\eqref{3.9} forms the analytical backbone of our time-reduction approach and serves as the foundation of the proposed computational framework. A central theoretical difficulty in deriving this system lies in justifying the approximation 
\[
\partial_t^2 \mathbf{u}(\mathbf{x}, t) \approx \sum_{n = 0}^N \mathbf{u}_n(\mathbf{x}) \Psi_n''(t).
\]
Establishing the validity of this representation is one of the principal analytical contributions of this work. The key insight involves the use of the Legendre polynomial--exponential basis, whose spectral properties and regularity enable effective control over the behavior of the series. 

This is formalized in Theorem~\ref{thm2.1}, which rigorously proves the convergence of the series 
\[
\sum_{n=0}^{\infty} \mathbf{u}_n(\mathbf{x}) \Psi_n''(t)
\]
in a suitable Sobolev space, and identifies the limit as the second time derivative $\partial_t^2 \mathbf{u}(\mathbf{x}, t)$. Without this spectral framework, such a convergence result would be difficult to obtain, as experienced in our earlier investigations. Theorem~\ref{thm2.1} therefore provides a crucial theoretical foundation for the entire time-dimensional reduction methodology.
\end{remark}

\begin{remark}
Recall that Proposition~\ref{prop2.1} guarantees that for every $n \geq 0$, the basis function $\Psi_n(t) = e^t Q_n(t)$ is infinitely differentiable on $(0, T)$, and none of its derivatives, including the second derivative $\Psi_n''(t)$, are identically zero. This nonvanishing property is essential: if $\Psi_n''(t)$ were identically zero for some $n$, the corresponding term $\mathbf{u}_n(\mathbf{x}) \Psi_n''(t)$ would vanish, and thus the influence of $\mathbf{u}_n$ would be completely omitted in the approximation of $\partial_{tt} \mathbf{u}(\mathbf{x}, t)$ in~\eqref{utt_appr}.

Such an omission would lead to an artificial loss of information and introduce unnecessary error into the numerical approximation. Therefore, the structural properties ensured by Proposition~\ref{prop2.1} are indispensable for the consistency and accuracy of the time reduction method.

This observation also highlights why more conventional bases, such as the classical Legendre polynomials or the trigonometric functions used in standard Fourier expansions, are not suitable in this context. These traditional bases may include functions whose second derivatives vanish identically. As a result, they cannot guarantee that each Fourier mode contributes meaningfully to approximations. The exponential-modulated Legendre basis $\{\Psi_n\}_{n \geq 0}$ avoids this issue, making it particularly well-suited for time-differential problems such as the one considered here.
\end{remark}

The core idea of the time dimension reduction method is to compute a finite collection of spatial modes
\begin{equation*}
U^N = \begin{bmatrix}\bu_0 & \mathbf{u}_1 & \dots & \mathbf{u}_N\end{bmatrix}^\top \in \left[H^{2+p}(\Omega)^d\right]^{N + 1},
\end{equation*}
which solves the reduced elliptic system~\eqref{3.9}. These spatial coefficients are then used to reconstruct a time-dependent approximation of the displacement field via the truncated expansion~\eqref{u_appr}. The resulting projection provides an estimate of the full space-time solution, from which the initial displacement and velocity can be inferred.

In the next section, we present the quasi-reversibility method to compute the solution to \eqref{3.9}.

\section{The convergence of the quasi-reversibility method} 
\label{sec_convergence}

Due to the approximation inherent in the representation \eqref{u_appr} and the presence of noise in the boundary data $ \mathbf{f} $ and $ \mathbf{g} $, the reduced system \eqref{2.9} may not admit an exact solution. To explicitly indicate the presence of noise, we denote the measured (perturbed) data by $ \mathbf{f}^{\delta} $ and $ \mathbf{g}^{\delta} $, corresponding to the ideal (noiseless) data $ \mathbf{f}^* $ and $ \mathbf{g}^* $, respectively. The term \emph{noise} refers to the deviation between these measured and true boundary values, and is quantified by the inequality
\begin{equation*}
\int_0^T \int_{\partial \Omega} \left( |\mathbf{f}^{\delta} - \mathbf{f}^*|^2 + |\mathbf{g}^{\delta} - \mathbf{g}^*|^2 \right) \,  {\rm d}\sigma(\mathbf{x}) \,  {\rm d}t \leq \delta^2,
\end{equation*}
where $ \delta > 0 $ denotes the prescribed noise level.
To address the potential nonexistence of a solution in the presence of such noise, we adopt a regularization strategy by formulating a Tikhonov-type variational problem. This regularized approach is commonly referred to as the \emph{quasi-reversibility method}, originally introduced in \cite{LattesLions:e1969}, and is particularly well-suited for stabilizing ill-posed inverse problems.
For $m \in \{1, \dots, N\},$ define the noisy boundary data of \eqref{3.9} as
\[
    {\bf f}_m^\delta(\x) = \int_0^T e^{-2t} {\bf f}^\delta(\x, t) \Psi_m(t)  {\rm d}t, 
    \quad
    {\bf g}_m^\delta(\x) = \int_0^T e^{-2t} {\bf g}^\delta(\x, t) \Psi_m(t)  {\rm d}t, 
\]

More precisely, let $p \geq 0$ be the number in Theorem~\ref{min_solution}, and define the functional \[J_{N, \eta, \delta} : [H^{2 + p}(\Omega)^d]^{N + 1} \to \mathbb{R}\] by
\begin{multline}
	 J_{N, \eta, \delta}(W) = \sum_{m=0}^N\Big\|\Div\left(
    \mathbb{C}:\nabla \mathbf{w}_m(\mathbf{x}) \right) - \sum_{n=1}^{N} s_{mn} \mathbf{w}_{n}(\mathbf{x})\Big\|^2_{[H^p(\Omega)^d]^{N + 1}}
    \\
    + \sum_{m=0}^N\int_{\partial\Omega} \left( |\mathbf{w}_m - \mathbf{f}_m^\delta|^2 + |\partial_\nu \mathbf{w}_m - \mathbf{g}_m^\delta|^2 \right) \,  {\rm d}\sigma(\mathbf{x}) 
    + \eta\sum_{m=0}^N \|\mathbf{w}_m\|^2_{H^{2 + p}(\Omega)^d},
     \label{J_functional}
\end{multline}
where $W = \begin{bmatrix} {\bf w}_0 & \mathbf{w}_1 & \dots& \mathbf{w}_N \end{bmatrix}^\top \in [H^{2 + p}(\Omega)^d]^{N + 1}$ and $\eta$ is the regularization parameter.

The functional $J_{N, \eta, \delta}$ defined in~\eqref{J_functional} consists of three key components: the first enforces the elliptic system derived from the time-reduction model~\eqref{3.9}; the second ensures consistency with the (potentially noisy) boundary measurements; and the third serves as a regularization term to stabilize the minimization process and promote selection of a minimal-norm solution. The regularization is measured in the Sobolev norm corresponding to the solution space specified in Theorem~\ref{min_solution}.

The following lemma reformulates the functional $J_{N, \eta, \delta}$ in space-time form. This representation facilitates the analysis and plays a key role in establishing the convergence theorem.

\begin{Lemma}
Let $W = \begin{bmatrix} \mathbf{w}_0 & \mathbf{w}_1 & \dots & \mathbf{w}_N \end{bmatrix}^\top \in \left[(H^{2+p}(\Omega)^d)\right]^{N+1}$, and let $\mathbf{w} = \mathbb{S}^N[W]$ denote its space-time expansion of $W$, see Definition \ref{def2.2} and \eqref{space_time_expansion}. Then the functional $J_{N, \eta, \delta}(W)$ admits the following integral form:
\begin{multline}
J_{N, \eta, \delta}(W) = 
\int_0^T e^{-2t}\left\| \Div(\mathbb{C} : \nabla \mathbf{w}(\mathbf{x}, t)) - \partial_{tt} \mathbf{w}(\mathbf{x}, t) \right\|^2_{H^p(\Omega)^d}  {\rm d}t 
\\
+ \int_{\Gamma_T}  e^{-2t}\left| \mathbf{w}(\mathbf{x}, t) - P^N \mathbf{f}^\delta(\mathbf{x}, t) \right|^2 \,  {\rm d}\sigma(\mathbf{x})  {\rm d}t
+ \int_{\Gamma_T}  e^{-2t}\left| \partial_\nu \mathbf{w}(\mathbf{x}, t) - P^N \mathbf{g}^\delta(\mathbf{x}, t) \right|^2 \,  {\rm d}\sigma(\mathbf{x})  {\rm d}t 
\\
+ \eta \left\| \mathbf{w} \right\|^2_{L^2_{ e^{-2t}}((0, T); H^{2+p}(\Omega)^d)},
\label{J_in_t}
\end{multline}
where $P^N$, as defined in \eqref{PN}, denotes the orthogonal projection onto the finite-dimensional subspace $\mathbb{V}_N \subset L^2_{ e^{-2t}}(0, T)$, defined in~\eqref{VN}.
\label{lem4.1}
\end{Lemma}

\begin{proof}
    Since $\{\Psi_n\}_{n = 0}^N$ is an orthonormal basis of the finite dimensional space $\mathbb{V}^N$ of $L^2(0, T)$, we have
    \begin{multline}
    \int_0^T e^{-2t}\left\|  \Div(\mathbb{C} : \nabla \mathbf{w}(\mathbf{x}, t)) - \partial_{tt} \mathbf{w}(\mathbf{x}, t) \right\|_{H^p(\Omega)^d}^2   {\rm d}t 
    \\
    = 
     \int_\Omega \sum_{m = 0}^N \big| \sum_{0 \leq |\alpha| \leq p}\langle D^{\alpha}[ \Div(\mathbb{C} : \nabla \mathbf{w}(\mathbf{x}, \cdot)) -  \partial_{tt} \mathbf{w}(\mathbf{x}, \cdot)], \Psi_m  \rangle_{L^2_{e^{-2t}}(0, T)}\big|^2  {\rm d}\x.
    \label{4.3}
    \end{multline}
   The right-hand side of \eqref{4.3} is exactly the first term in the definition of $J_{N, \eta, \delta}(W)$ in \eqref{J_functional}.
    Similarly, we can use the same argument to match the other terms of \eqref{J_in_t} and their corresponding terms on the right-hand side of \eqref{J_functional}. We therefore obtain \eqref{J_in_t}.
    \end{proof}

Let $U^{N, \eta, \delta}_{\min} = 
\begin{bmatrix}
    \bu_0^{N, \eta, \delta} & \bu_1^{N, \eta, \delta} &\dots & \bu_N^{N, \eta, \delta}
\end{bmatrix}^\top$ denote the minimizer of the functional $J_{N, \eta, \delta}$. Once this minimizer is computed, its space-time reconstruction $\mathbb{S}^N [U^{N, \eta, \delta}_{\min}]$ provides an approximate solution $\bu^*$ to the boundary value problem~\eqref{2.1111}, and thus $\mathbb{S}^N [U^{N, \eta, \delta}_{\min}](\x, 0)$ and $\partial_t \mathbb{S}^N [U^{N, \eta, \delta}_{\min}](\x, 0)$ yields a numerical approximation to the solution of the inverse problem posed in Problem~\ref{isp}.
This solving procedure naturally raises the following fundamental questions:
\begin{enumerate}
    \item Does the minimizer $U^{N, \eta, \delta}_{\min}$ exist and is it unique?
    \item Does the reconstructed approximation
    \begin{equation*}
        \mathbf{u}^{N, \delta, \eta}_{\min}(\x, t) := \mathbb{S}^N[U^{N, \delta, \eta}_{\min}](\x, t) = \sum_{n = 0}^N \mathbf{u}_n^{N, \delta, \eta}(\x) \Psi_n(t), \quad (\x, t) \in \Omega_T
    \end{equation*}
    converge to the true solution $\mathbf{u}$ of~\eqref{2.1111} as $N \to \infty$ and $\eta, \delta \to 0$?
\end{enumerate}

We address these questions in the following theorem.
\begin{Theorem}
Let \( p \geq 0 \). Assume that the system \eqref{2.1111} has 
the unique minimum solution $u^*\in L^2(0,T; H^{2+p}(\Omega)^d)\cap H^2(0,T; H^p(\Omega)^d)$
 as in Theorem~\ref{min_solution} and that the series
 $$\sum_{n=0}^\infty \langle \mathbf{u}^*(\mathbf{x},.),\Psi_n\rangle_{L^2_{e^{-2t}}(0,T)}\Psi_n''(t) $$ 
 converges in $L^2(0,T; H^p(\Omega)^d)$.
 Then the following statements hold:

\begin{enumerate}
\item (Existence and uniqueness of the minimizer)  
Assume that the boundary data sequences \( (\mathbf{f}_m^{\delta})_{m=0}^N \) and \( (\mathbf{g}_m^{\delta})_{m=0}^N \) belong to \( ([L^2(\partial \Omega)]^d)^{N + 1} \). Then the functional \( J_{N, \eta, \delta} \) admits a unique minimizer \( U^{N, \eta, \delta}_{\min} \in [(H^{2+p}(\Omega))^d]^{N + 1} \), satisfying
\[
U^{N, \eta, \delta}_{\min} = \operatorname*{argmin}_{W \in [(H^{2+p}(\Omega))^d]^{N + 1}} J_{N, \eta, \delta}(W).
\]

\item (Approximation of the second derivative)  
For any fixed noise level \( \delta > 0 \), there exists a threshold \( N(\delta) \in \mathbb{N} \) such that for all \( N \geq N(\delta) \), the projection \( P^N \mathbf{u}^* \) of the minimal norm solution \( \mathbf{u}^* \) satisfies
\begin{equation}
\left\| P^N\mathbf{u}^*_{tt} - (P^N \mathbf{u}^*)_{tt} \right\|_{L^2_{e^{-2t}}((0, T); H^p(\Omega)^d)}^2 \leq \delta^2.
\label{4.4}
\end{equation}

\item (Convergence of the regularized solution)  
Assume that the regularization parameter \( \eta = \eta(\delta) \) satisfies \( \eta(\delta) \to 0 \) and \( \delta^2 = o(\eta(\delta)) \) as \( \delta \to 0 \). Define the admissible parameter set
\[
\Theta := \left\{ (N, \eta, \delta) : N \geq N(\delta),\ \delta^2 = o(\eta) \right\}.
\]
Then,
\begin{multline}
    \lim_{\Theta \ni (N, \eta, \delta) \to (\infty, 0^+, 0^+)}  \Big[
    \big\|
        \mathbb{S}^N[U^{N, \eta, \delta}_{\min}] - \bu^*
    \big\|_{L^2_{e^{-2t}}((0, T); H^p(\Omega)^d)}
    \\
    +
    \big\|
        \partial_t \mathbb{S}^N[U^{N, \eta, \delta}_{\min}] - \bu^*_t
    \big\|_{L^2_{e^{-2t}}((0, T); H^p(\Omega)^d)}
    \Big] = 0.
    \label{main_convergence}
\end{multline}


\end{enumerate}
\label{main-theorem}
\end{Theorem}

Before proving Theorem \ref{main-theorem}, we prove the following Lemma.
\begin{Lemma}
Let $T > 0$. There exists a constant $C > 0$ such that
\[
\|v'\|^2_{L^2_{e^{-2t}}(0, T)} \leq C \left( \|v\|^2_{L^2_{e^{-2t}}(0, T)} + \|v''\|^2_{L^2_{e^{-2t}}(0, T)} \right)
\]
for all $v \in H^2(0, T)$.
\label{lem4.2}
\end{Lemma}

\begin{proof}
Suppose, for contradiction, that the inequality does not hold. Then for each $n \in \mathbb{N}$, there exists a function $v_n \in H^2(0, T)$ such that
\[
\|v_n'\|^2_{L^2_{e^{-2t}}(0, T)} \geq n \left( \|v_n\|^2_{L^2_{e^{-2t}}(0, T)} + \|v_n''\|^2_{L^2_{e^{-2t}}(0, T)} \right).
\]
Define the normalized sequence $w_n := v_n / \|v_n'\|_{L^2_{e^{-2t}}(0, T)}$. Then $\|w_n'\|_{L^2_{e^{-2t}}(0, T)} = 1$ and
\[
n \left( \|w_n\|^2_{L^2_{e^{-2t}}(0, T)} + \|w_n''\|^2_{L^2_{e^{-2t}}(0, T)} \right) \leq 1.
\]
This implies that $\|w_n\|_{H^2(0, T)}^2 = \|w_n\|^2_{L^2(0, T)} + \|w_n'\|^2_{L^2(0, T)} + \|w_n''\|^2_{L^2(0, T)} \leq 2e^{2T}$.

Hence, $\{w_n\}$ is bounded in $H^2(0, T)$. By the compact embedding $H^2(0, T) \hookrightarrow H^1(0, T)$, there exists a subsequence $\{w_{n_k}\}$ and a function $w \in H^2(0, T)$ such that
\begin{align*}
w_{n_k} &\rightharpoonup w \quad \text{weakly in } H^2(0, T), \\
w_{n_k} &\to w \quad \text{strongly in } H^1(0, T).
\end{align*}

As a result, 
\[
\|w_{n_k}\|_{L^2_{e^{-2t}}(0, T)} \to \|w\|_{L^2_{e^{-2t}}(0, T)}, \quad 
\|w_{n_k}'\|_{L^2_{e^{-2t}}(0, T)} \to \|w'\_{L^2_{e^{-2t}}(0, T)}.
\]

However, from the normalization and the inequality above, we have $\|w_n'\|_{L^2_{e^{-2t}}(0, T)} = 1$ and $\|w_n\|_{L^2_{e^{-2t}}(0, T)} \leq 1/\sqrt{n} \to 0$ as $n \to \infty$. Therefore, $\|w\|_{L^2_{e^{-2t}}(0, T)} = 0$ and $\|w'\|_{L^2_{e^{-2t}}(0, T)} = 1$. This implies $w = 0$ and $w' \neq 0$, which is a contradiction.
\end{proof}

\begin{proof}[Proof of Theorem \ref{main-theorem}]
    The first statement of the theorem is a direct consequence of the classical Lax-Milgram theorem, and we omit the standard details.

To prove the second part, we invoke Theorem~\ref{thm2.1}. It is straightforward to verify that the projection $P^N \mathbf{u}_{tt}^*$ converges to $\mathbf{u}_{tt}^*$ in $L^2((0, T); H^p(\Omega)^d)$ as $N \to \infty$. Moreover, by Theorem~\ref{thm2.1}, the second time derivative $\partial_{tt} P^N \mathbf{u}^*$ also converges to $\mathbf{u}_{tt}^*$ in the same norm. Consequently, we obtain the estimate
\[
\lim_{N \to \infty} \int_0^Te^{-2t} \left\| P^N \mathbf{u}_{tt}^*(\mathbf{x}, t) - \partial_{tt} P^N \mathbf{u}^*(\mathbf{x}, t) \right\|^2_{H^p(\Omega)^d} \,  {\rm d}t = 0.
\]
Therefore, for any $\delta > 0$, there exists a threshold $N(\delta) > 0$ such that
\begin{equation}
\int_0^T e^{-2t}\left\| P^N \mathbf{u}_{tt}^*(\mathbf{x}, t) - \partial_{tt} P^N \mathbf{u}^*(\mathbf{x}, t) \right\|^2_{H^p(\Omega)^d} \,  {\rm d}t \leq \delta^2
\quad \text{for all } N \geq N(\delta),
\label{4.6}
\end{equation}
which implies the estimate~\eqref{4.4}.

    We now proceed to prove the third part of the theorem. 
    For convenience, from this point onward, let $C$ denote a generic positive constant that may vary from line to line but depends only on $T$, $d$, and the domain $\Omega$.

	For each $N \geq 0$, let $ \bu^*_N =  \begin{bmatrix}
            \bu^*_0 & \bu^*_1 & \dots & \bu^*_N
        \end{bmatrix}^{\top}$ be the vector of Fourier modes of $\bu^*$ with respect to the Legendre polynomial-exponential basis. Let $P^N \bu^*$ and $P^N \bu^*_{tt}$ denote the projections of $\bu^*$ and $\bu^*_{tt}$ onto the subspace $\mathbb{V}_N$ with respect to the time variable, as defined in Definition~\ref{def2.2} and equation~\eqref{PN}.
	 
    It is obvious that the space-time expansion  $\mathbb{S}^N[\bu^*_N]$ is exactly $ P^N \bu^*$. By the definition of the mismatch functional $J_{N, \eta, \delta}$ in \eqref{J_functional} and by Lemma \ref{lem4.1},
    we have
    \begin{align}
        J_{N, \eta, \delta}(U^{N, \eta, \delta}_{\rm min}) &\leq J_{N, \eta, \delta}(\bu^*_N) \notag\\
        &= \int_0^T e^{-2t}\| P^N \Div (\mathbb{C}:\nabla  \bu^*) - \partial_{tt} P^N \bu^*\|^2_{H^p(\Omega)^d}  {\rm d}t  
        + \int_{\Gamma_T}e^{-2t} \left| P^N \bu^*(\mathbf{x}, t) - P^N \mathbf{f}^\delta(\mathbf{x}, t) \right|^2 \,  {\rm d}\sigma(\mathbf{x})  {\rm d}t
        \notag
        \\
&\quad + \int_{\Gamma_T} e^{-2t}\left| \partial_\nu P^N \bu^*(\mathbf{x}, t) - P^N \mathbf{g}^\delta(\mathbf{x}, t) \right|^2 \,  {\rm d}\sigma(\mathbf{x})  {\rm d}t 
+ \eta \left\| P^N \bu^* \right\|^2_{L^2_{e^{-2t}}((0, T); H^{2+p}(\Omega)^d)}
\notag
\\
&\leq
	\int_{0}^Te^{-2t} \| P^N \bu_{tt}^*(\x, t) - \partial_{tt} P^N \bu^*\|^2_{H^p(\Omega)^d}   {\rm d}t 
	+ \delta^2 
	+ \eta \|\bu^*\|_{L^2_{e^{-2t}}((0, T); H^{2 + p}(\Omega)^d)}^2.
	\label{4.5}
    \end{align}
    
It follows from Lemma \ref{lem4.1}, \eqref{J_functional}, \eqref{4.5} and \eqref{4.6} that
\begin{equation}
	\eta \|\mathbb{S}^N[U^{N, \eta, \delta}_{\rm min}]\|_{L^2_{e^{-2t}}((0, T); H^{2 + p}(\Omega)^d)}^2 \leq J(U^{N, \eta, \delta}_{\rm min})
	\leq 2\delta^2 + \eta \|\bu^*\|_{L^2_{e^{-2t}}((0, T); H^{2 + p}(\Omega)^d)}^2
    \label{Estimate_J_U_min}
\end{equation}
provided that $N \geq N(\delta).$ 
Thus, using the choice of $(N, \eta, \delta) \in \Theta$ with $\eta > \delta^2$, we have
\begin{equation}
	\|\mathbb{S}^N[U^{N, \eta, \delta}_{\rm min}]\|_{L^2_{e^{-2t}}((0, T); H^{2 + p}(\Omega)^d)}^2
	\leq \frac{2\delta^2}{\eta} + \|\bu^*\|_{L^2_{e^{-2t}}((0, T); H^{2 + p}(\Omega)^d)}^2.
	\label{4.7}
\end{equation}
By the triangle inequality, we have
\begin{align}
	\|\partial_{tt}  \mathbb{S}^N[U^{N, \eta, \delta}_{\rm min}]\|_{L^2_{e^{-2t}}((0, T); H^p(\Omega)^d)} 	
	\notag
	&\leq 
	\|\Div(\mathbb{C} : \nabla (\mathbb{S}^N[U^{N, \eta, \delta}_{\rm min}])) - \partial_{tt}\mathbb{S}^N[U^{N, \eta, \delta}_{\rm min}]\|_{L^2_{e^{-2t}}((0, T); H^p(\Omega)^d)} 
	\notag
	\\
	&\hspace{4cm} +\|\Div(\mathbb{C} : \nabla [\mathbb{S}^N[U^{N, \eta, \delta}_{\rm min}]])\|_{L^2_{e^{-2t}}((0, T); H^p(\Omega)^d)} \notag
	\\
	&\leq J_{N, \eta, \delta}(U^{N, \eta, \delta}_{\rm min}) + C\|\mathbb{S}^N[U^{N, \eta, \delta}_{\rm min}]\|_{L^2_{e^{-2t}}((0, T); H^{2 + p}(\Omega)^d)}
	\label{4.8}
\end{align}
for all $(N, \eta, \delta)$ in $\Theta$.
Combining \eqref{4.5}, \eqref{4.6}, \eqref{4.7}, and \eqref{4.8} gives
\begin{multline}
	\|\partial_{tt}  \mathbb{S}^N[U^{N, \eta, \delta}_{\rm min}]\|_{L^2_{e^{-2t}}((0, T); H^{p}(\Omega)^d)} 
	+
	\|\mathbb{S}^N[U^{N, \eta, \delta}_{\rm min}]\|_{L^2_{e^{-2t}}((0, T); H^{2 + p}(\Omega)^d)}
	\\
	\leq 
	\frac{2\delta^2}{\eta} 
	+ \|\bu^*\|_{L^2_{e^{-2t}}((0, T); H^{2 + p}(\Omega)^d} \leq C \|\bu^*\|_{L^2_{e^{-2t}}((0, T); H^{2 + p}(\Omega)^d}
    \label{4.8888}
\end{multline}
for all $(N, \eta, \delta)$ in $\Theta$.
Applying Lemma \ref{lem4.2} and using \eqref{4.8888}, we have
\begin{equation}
    \|\partial_{t}  \mathbb{S}^N[U^{N, \eta, \delta}_{\rm min}]\|_{L^2_{e^{-2t}}((0, T); H^{p}(\Omega)^d)} \leq C\|\bu^*\|_{L^2_{e^{-2t}}((0, T); H^{2 + p}(\Omega)^d}.
    \label{4.9}
\end{equation}
Due to \eqref{4.7}, \eqref{4.8888}, and \eqref{4.9}, the set 
\[
	\Big\{\mathbb{S}^N [U^{N, \eta, \delta}_{\rm min}]: (N, \eta, \delta) \in\Theta \Big\}
\] is bounded in $L^2((0, T); H^{2 + p}(\Omega)^d) \cap H^2((0, T); H^p(\Omega)^d)$.
Using the compactness argument, we can assume, without loss of generality, that 
\begin{align*}    
    &\mathbb{S}^N[U^{N, \eta, \delta}_{\rm min}] \to {\bf z} \quad \mbox{strongly in } {L^2_{e^{-2t}}((0, T); H^{p}(\Omega)^d)},\\
    &\partial_t \mathbb{S}^N[U^{N, \eta, \delta}_{\rm min}] \to {\bf z}_t \quad \mbox{strongly in } {L^2_{e^{-2t}}((0, T); H^{p}(\Omega)^d)},\\
     &\mathbb{S}^N[U^{N, \eta, \delta}_{\rm min}] \rightharpoonup  {\bf z} \quad \mbox{weakly in } {L^2_{e^{-2t}}((0, T); 
     L^2(\partial\Omega)^d)},\\
    &\partial_{\nu} \mathbb{S}^N[U^{N, \eta, \delta}_{\rm min}] \rightharpoonup \partial_\nu {\bf z} \quad \mbox{weakly in } {L^2_{e^{-2t}}((0, T); L^2(\partial\Omega)^d)},
\end{align*}
as $\Theta \ni (N, \eta, \delta) \to (\infty, 0^+, 0^+)$ for some vector-valued function $\bf{z}.$ 
Due to \eqref{J_in_t} and \eqref{Estimate_J_U_min}, the vector-valued function ${\bf z}$ satisfies 
\begin{multline*}
    \int_0^T e^{-2t}\left\| \Div(\mathbb{C} : \nabla \mathbf{z}(\mathbf{x}, t)) - \partial_{tt} \mathbf{z}(\mathbf{x}, t) \right\|^2_{H^p(\Omega)^d}  {\rm d}t 
\\
+ \int_{\Gamma_T} e^{-2t}\left| \mathbf{z}(\mathbf{x}, t) -  \mathbf{f}^*(\mathbf{x}, t) \right|^2 \,  {\rm d}\sigma(\mathbf{x})  {\rm d}t
+ \int_{\Gamma_T} \left| \partial_\nu \mathbf{z}(\mathbf{x}, t) -  \mathbf{g}^*(\mathbf{x}, t) \right|^2 \,  {\rm d}\sigma(\mathbf{x})  {\rm d}t 
\leq 0.
\end{multline*}
Thus, ${\bf z}$ satisfies \eqref{2.1111}. 
Due to \eqref{4.7} and the convergence of $\mathbb{S}^N[U^{N, \eta, \delta}_{\rm min}]$ to $\bf z$, 
\[
\|{\bf z}\|_{L^2_{e^{-2t}}((0, T); H^{2 + p}(\Omega)^d)} \leq  \|{\bf u}^*\|_{L^2_{e^{-2t}}((0, T); H^{2 + p}(\Omega)^d)}.
\]
By the minimum norm property of $\bu^*$, ${\bf z} = \bu^*.$

\end{proof}

\begin{remark}
	\begin{enumerate}
	    \item Although our theoretical framework does not require uniqueness, it is highly plausible that problem~\eqref{2.1111} admits a unique solution due to the presence of Cauchy boundary data. In such cases, the time-dimensional reduction method not only produces a stable approximation but also guarantees convergence to the true physical solution. This reinforces the practical reliability of our approach in realistic settings where uniqueness is expected.

	    \item Importantly, the convergence result in Theorem~\ref{main-theorem} remains valid even when the boundary data for the inverse problem are available only on a proper subset of the boundary. That is, if the measurements are restricted to
	    \begin{equation}
	    	\mathbf{f}(\mathbf{x}, t) = \mathbf{u}(\mathbf{x}, t), \quad
	    	\mathbf{g}(\mathbf{x}, t) = \partial_\nu \mathbf{u}(\mathbf{x}, t)
	    	\quad \text{for } (\mathbf{x}, t) \in \Gamma \times [0, T],
	    	\nonumber
	    \end{equation}
	    where \( \Gamma \subset \partial \Omega \), the proposed method and the convergence analysis remain applicable. This flexibility significantly enhances the scope of the method in practical inverse problems, where full boundary data are often unavailable.
	\end{enumerate}
\end{remark}

\begin{remark}
    A direct consequence of \eqref{main_convergence} is the following:
\begin{align*}
    \mathbb{S}^N [U^{N, \eta, \delta}_{\min}](\cdot, 0) &\to \bu^*(\cdot, 0) \quad \text{strongly in } H^p(\Omega)^d, \\
    \partial_t \mathbb{S}^N [U^{N, \eta, \delta}_{\min}](\cdot, 0) &\rightharpoonup \bu^*_t(\cdot, 0) \quad \text{weakly in } H^p(\Omega)^d,
\end{align*}
as $\Theta \ni (N, \eta, \delta) \to (\infty, 0^+, 0^+)$.  
These convergence results guarantee that the initial displacement field $\mathbf{p}$ and the initial velocity field $\mathbf{q}$ can be approximated by $\mathbb{S}^N [U^{N, \eta, \delta}_{\min}](\cdot, 0)$ and $\partial_t \mathbb{S}^N [U^{N, \eta, \delta}_{\min}](\cdot, 0)$, respectively.
\end{remark}

\section{Numerical study} \label{sec_num}

We numerically solve Problem \ref{isp} in two-dimensional case ($ d = 2 $) and define the computational domain as $ \Omega = (-1, 1)^2 $. To generate boundary data on $ \partial\Omega$ , we need to solve the elastic wave equation in \eqref{main_ivp}. Since computations over the unbounded domain $\mathbb{R}^2$ are not feasible, we approximate it by a sufficiently large bounded domain $ G = (-3, 3)^2 $. The forward problem is numerically solved on this larger domain using an explicit finite difference scheme, and the resulting solution is then restricted to the original domain $ \Omega $.
To carry out the discretization, we construct a uniform spatial grid $ \mathcal{G} $ consisting of $ 121 \times 121 $ points on $ G $. For the temporal discretization, we set the final time $ T = 1 $ and define a uniform partition $ \mathcal{T} $ consisting of 6400 time steps. The details of the numerical steps using an explicit finite difference scheme for solving the forward problem on the discrete domain $ \mathcal{G} \times \mathcal{T} $, as well as the domain truncation and restriction procedures, are standard in the literature and are therefore omitted here for brevity.

Assuming the forward simulation has been carried out, let $\mathbf{u}^*$ denote the exact (noiseless) solution. The corresponding boundary data are extracted as
\[
\mathbf{f}^* = \mathbf{u}^*|_{\partial \Omega \times [0, T]}, \quad 
\mathbf{g}^* = \partial_{\nu} \mathbf{u}^*|_{\partial \Omega \times [0, T]}.
\]
To simulate measurement noise, the noisy data are then defined by
\begin{equation}
\mathbf{f} = \mathbf{f}^* (1 + \delta\, \mathrm{rand}), \quad 
\mathbf{g} = \mathbf{g}^* (1 + \delta\, \mathrm{rand}),
\end{equation}
where $\delta = 10\%$, and $\mathrm{rand}$ is a function generating uniformly distributed random values in the interval $[-1, 1]$. For simplicity in implementation, we set $p = 0$ in our computation. In this case, $H^p(\Omega)$ is understood as $L^2(\Omega).$

With the boundary data available, we proceed to solve the inverse problem. The key computational steps of our proposed method, based on the time dimension reduction framework, are outlined in Algorithm~\ref{alg}.
\begin{algorithm}[h!]
\caption{\label{alg}Solution procedure for Problem~\ref{isp} via the time dimension reduction method}
\begin{algorithmic}[1]
\State Select a truncation parameter $N$ and a regularization parameter $\eta$.
\State Solve the minimization problem for the strictly convex functional $J_{N, \eta, \delta}$ defined in \eqref{J_functional}. Let the minimizer be denoted by
\[
U^{\rm comp} = \begin{bmatrix} 
\bu_0^{\rm comp} & \bu_1^{\rm comp} & \dots & \bu_N^{\rm comp}
\end{bmatrix}^\top.
\]
\State Reconstruct the displacement field $\bu^{\rm comp}$ using the truncated series expansion:
\[
\bu^{\rm comp}(\mathbf{x}, t) = \mathbb{S}^N[U^{\rm comp}] = \sum_{n=0}^N \bu_n^{\rm comp}(\mathbf{x}) \Psi_n(t),
\]
for all $(\mathbf{x}, t) \in \Omega_T$.
\State Compute the approximate initial displacement field:
\[
\mathbf{p}^{\rm comp}(\mathbf{x}) = \bu^{\rm comp}(\mathbf{x}, 0),
\]
for all $\mathbf{x} \in \Omega$.
\State Compute the approximate initial velocity field:
\[
\mathbf{q}^{\rm comp}(\mathbf{x}) = \partial_t \bu^{\rm comp}(\mathbf{x}, 0),
\]
for all $\mathbf{x} \in \Omega$.
\end{algorithmic}
\end{algorithm}

\subsection{Numerical implementation of Algorithm~\ref{alg}}

In this subsection, we detail the numerical implementation of the proposed Algorithm~\ref{alg}, including the selection of parameters, solution of the optimization problem, and practical considerations relevant to reproducing the results.

\textit{Step 1}. The parameters $ N $ and $ \eta $ are selected through a trial-and-error process, using Test~1 in Subsection \ref{subsec_example} as a reference case. Specifically, these values are tuned such that the numerical reconstruction for Test~1 yields an acceptable level of accuracy. Once satisfactory performance is achieved in this benchmark scenario, the same parameters are fixed and subsequently applied to all other test cases without further adjustment. This strategy ensures consistency in evaluating the robustness of the proposed method across different scenarios. In our computation, we set $ N = 30 $ and $ \eta = 10^{-6} $.

\textit{Step 2}. Each integrand appearing in the definition of the functional $ J_{N, \eta, \delta} $ in \eqref{J_functional} is the square of an affine function with respect to the unknown vector $W $. Consequently, minimizing $ J_{N, \eta, \delta} $ reduces to solving a standard least-squares optimization problem. In our implementation, this is carried out using MATLAB's built-in function \texttt{lsqlin}. This step is straightforward to execute. For guidance on the syntax and input format of this command, users are referred to MATLAB's documentation, which can be accessed by entering \texttt{help lsqlin} in the MATLAB command window.

The implementation of the remaining steps is straightforward.

\subsection{Numerical examples} \label{subsec_example}

We present 3 numerical tests.

\subsubsection*{Test 1}
In this test case, the elasticity tensor $\mathbb{C} = (C_{ijkl})_{i,j,k,l=1}^2$ represents a homogeneous and isotropic elastic medium, and is defined component-wise by
\[
C_{ijkl} = \lambda \, \delta_{ij} \delta_{kl} + \mu \left( \delta_{ik} \delta_{jl} + \delta_{il} \delta_{kj} \right),
\]
for all indices $i, j, k, l \in \{1, 2\}$, where $\delta_{ij}$ denotes the Kronecker delta. The Lam\'e parameters are set to $\mu = 2$ and $\lambda = 1$.

The first and second components of the true initial displacement field ${\bf p}^{\rm true}$ are given respectively by
\[
p_1^{\rm true}(x, y) =
\begin{cases}
1 & \text{if } 3(x - 0.5)^2 + y^2 \leq 0.6^2, \\
0 & \text{otherwise},
\end{cases}
\quad
p_2^{\rm true}(x, y) =
\begin{cases}
1 & \text{if } x^2 + 3(y - 0.5)^2 \leq 0.6^2, \\
0 & \text{otherwise}.
\end{cases}
\]
The first and second components of the true initial velocity field ${\bf q}^{\rm true}$ are defined respectively as
\[
q_1^{\rm true}(x, y) =
\begin{cases}
30 & \text{if } 3(x - 0.5)^2 + y^2 \leq 0.6^2, \\
0 & \text{otherwise},
\end{cases}
\quad
q_2^{\rm true}(x, y) =
\begin{cases}
30 & \text{if } x^2 + 3(y + 0.4)^2 \leq 0.3^2, \\
0 & \text{otherwise}.
\end{cases}
\] 

\begin{figure}[h!]
\centering
\subfloat[$ p_1^{\rm true} $]{
\includegraphics[width=0.23\textwidth]{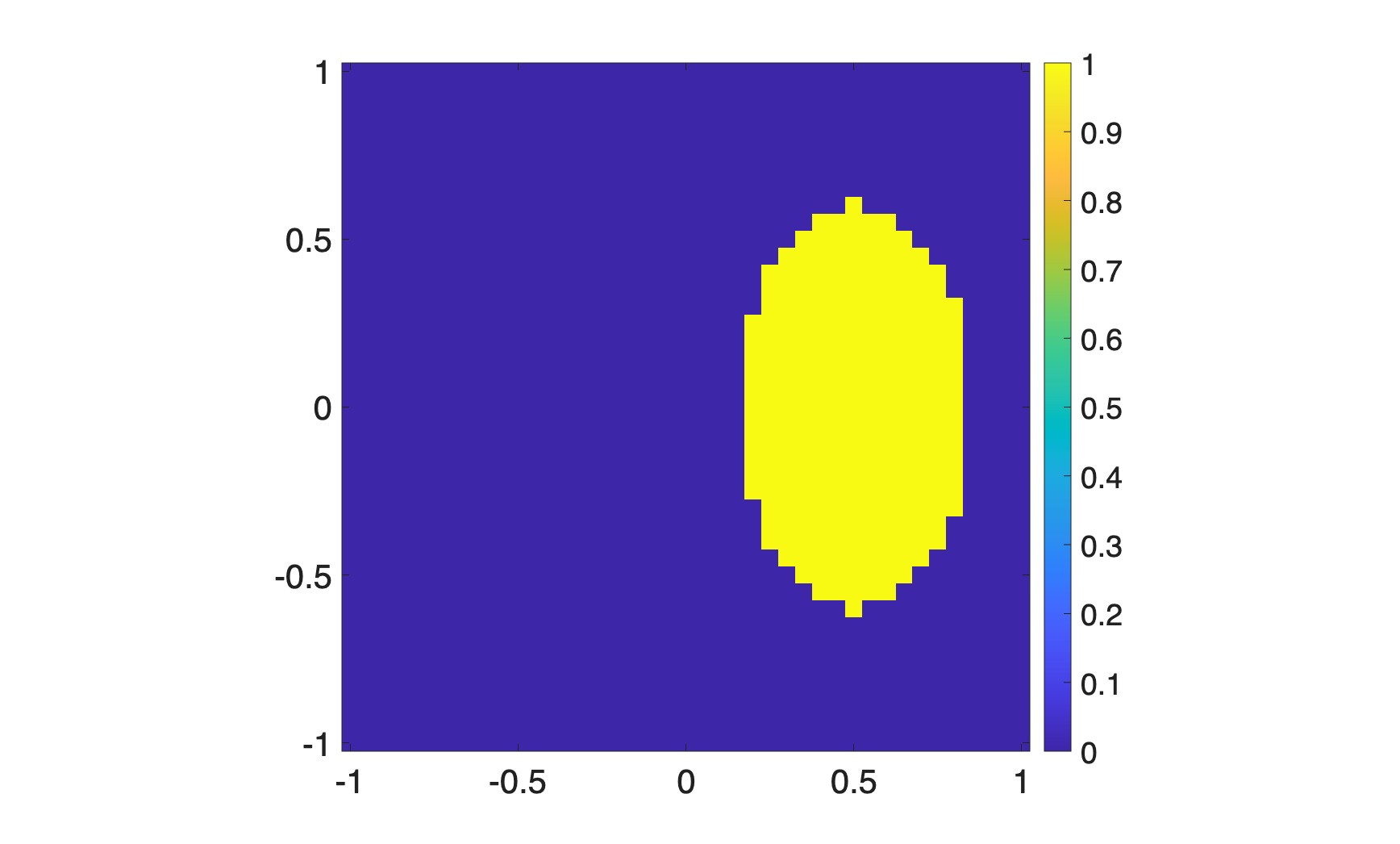}
}
\hfill
\subfloat[ $ p_2^{\rm true} $]{
\includegraphics[width=0.23\textwidth]{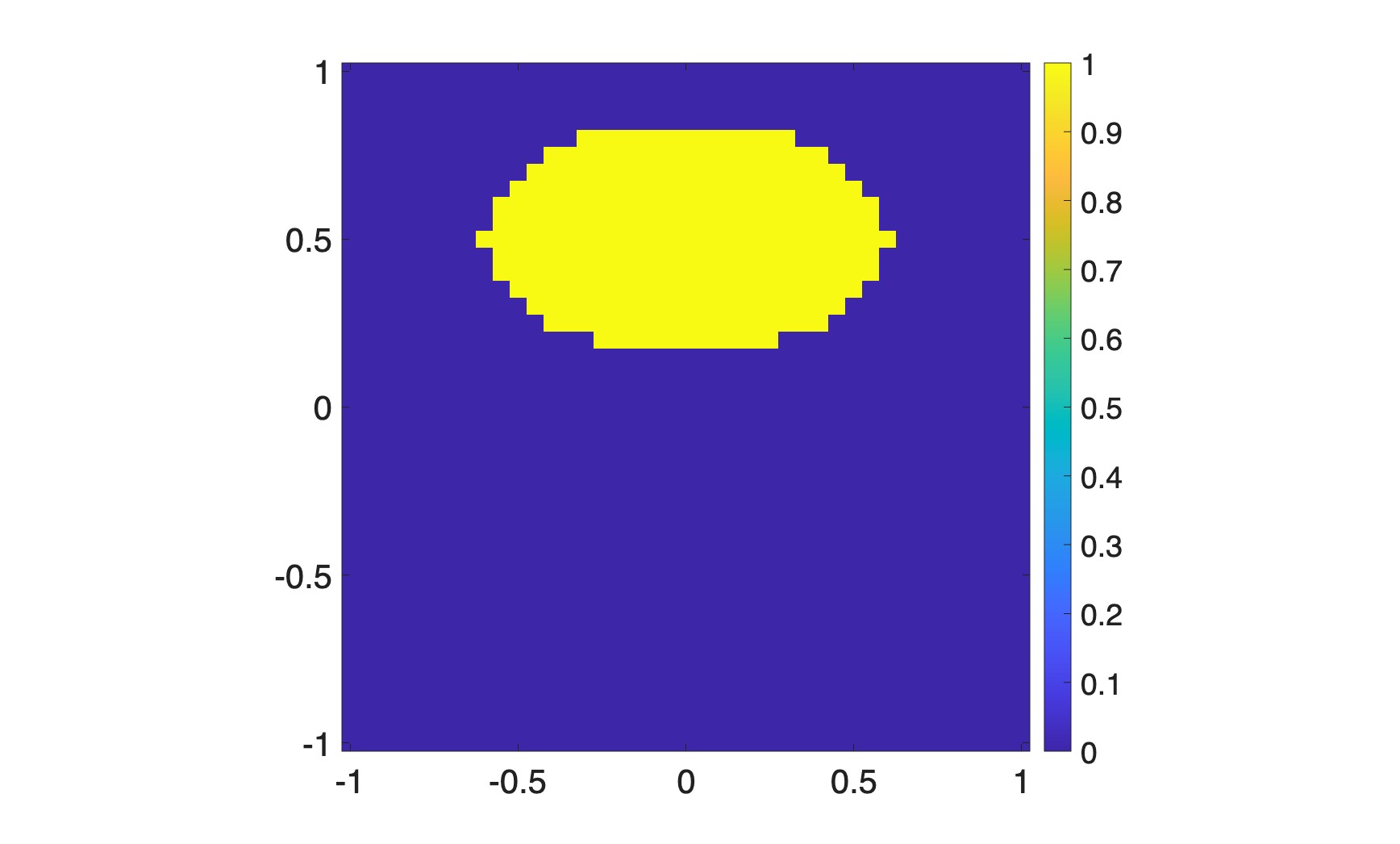}
}
\hfill \subfloat[ $ q_1^{\rm true} $]{
\includegraphics[width=0.23\textwidth]{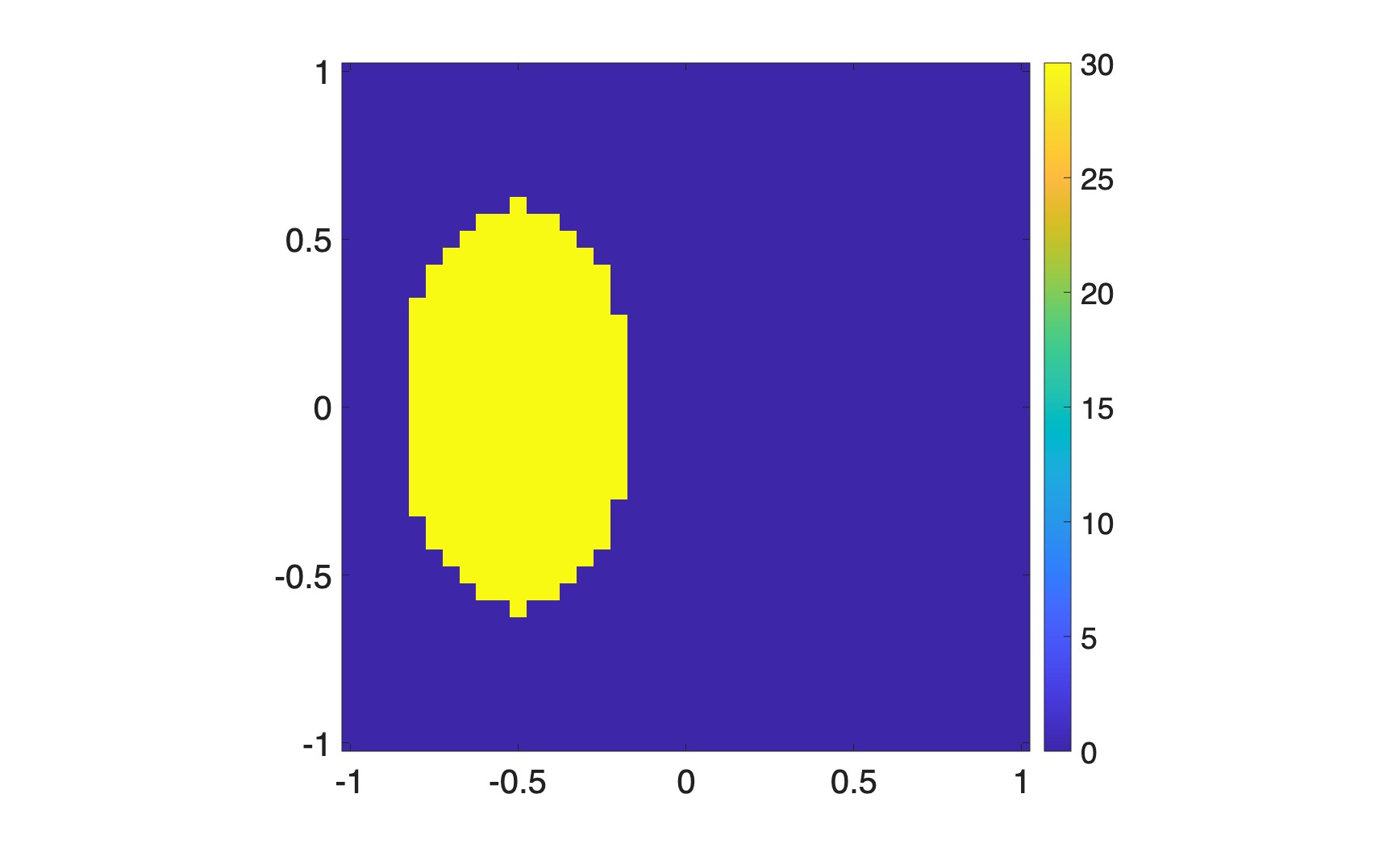}
}
\hfill \subfloat[ $ q_2^{\rm true} $]{
\includegraphics[width=0.23\textwidth]{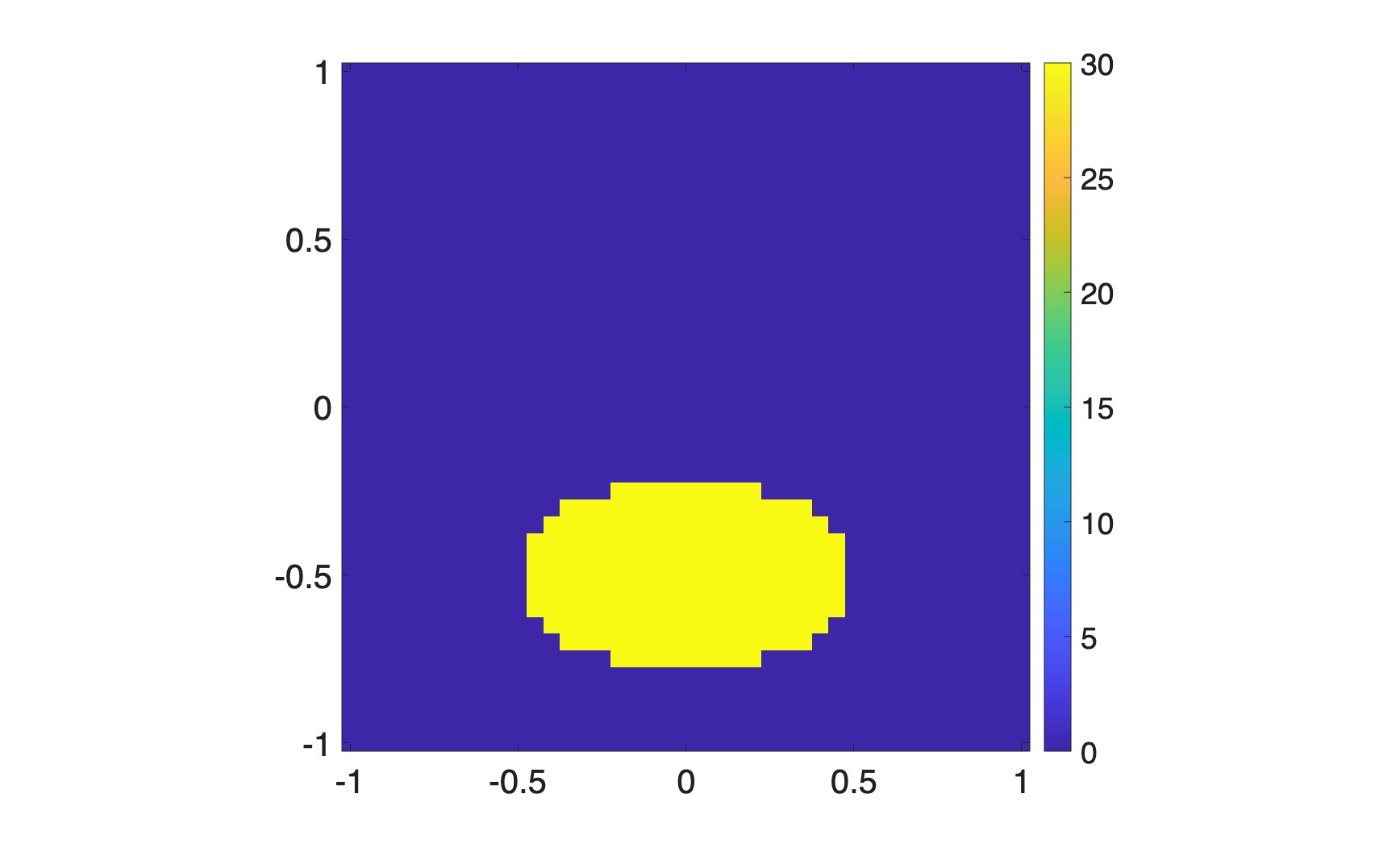}
}

\subfloat[ $ p_1^{\rm comp} $]{
\includegraphics[width=0.23\textwidth]{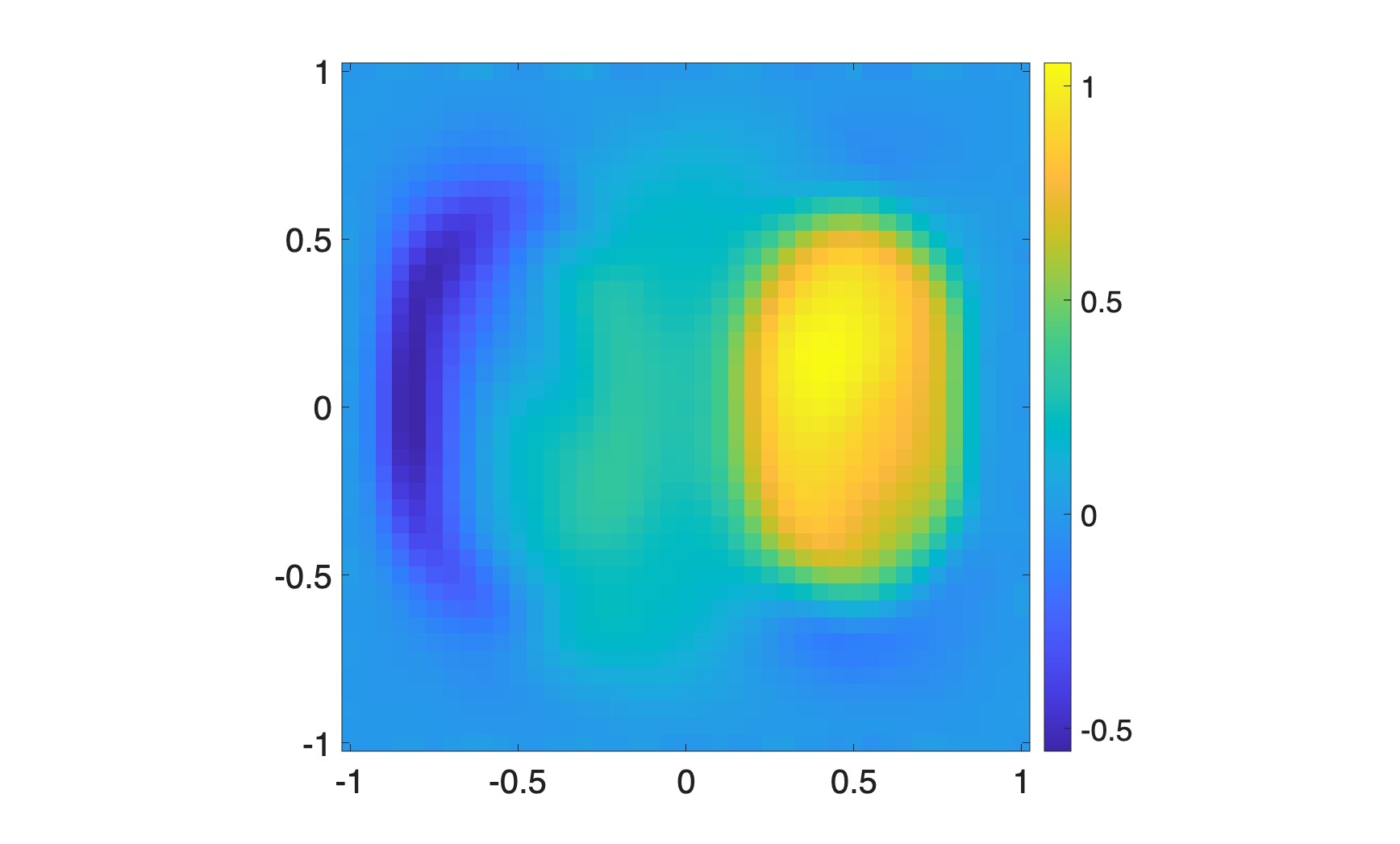}
}
\hfill
\subfloat[ $ p_2^{\rm comp} $]{
\includegraphics[width=0.23\textwidth]{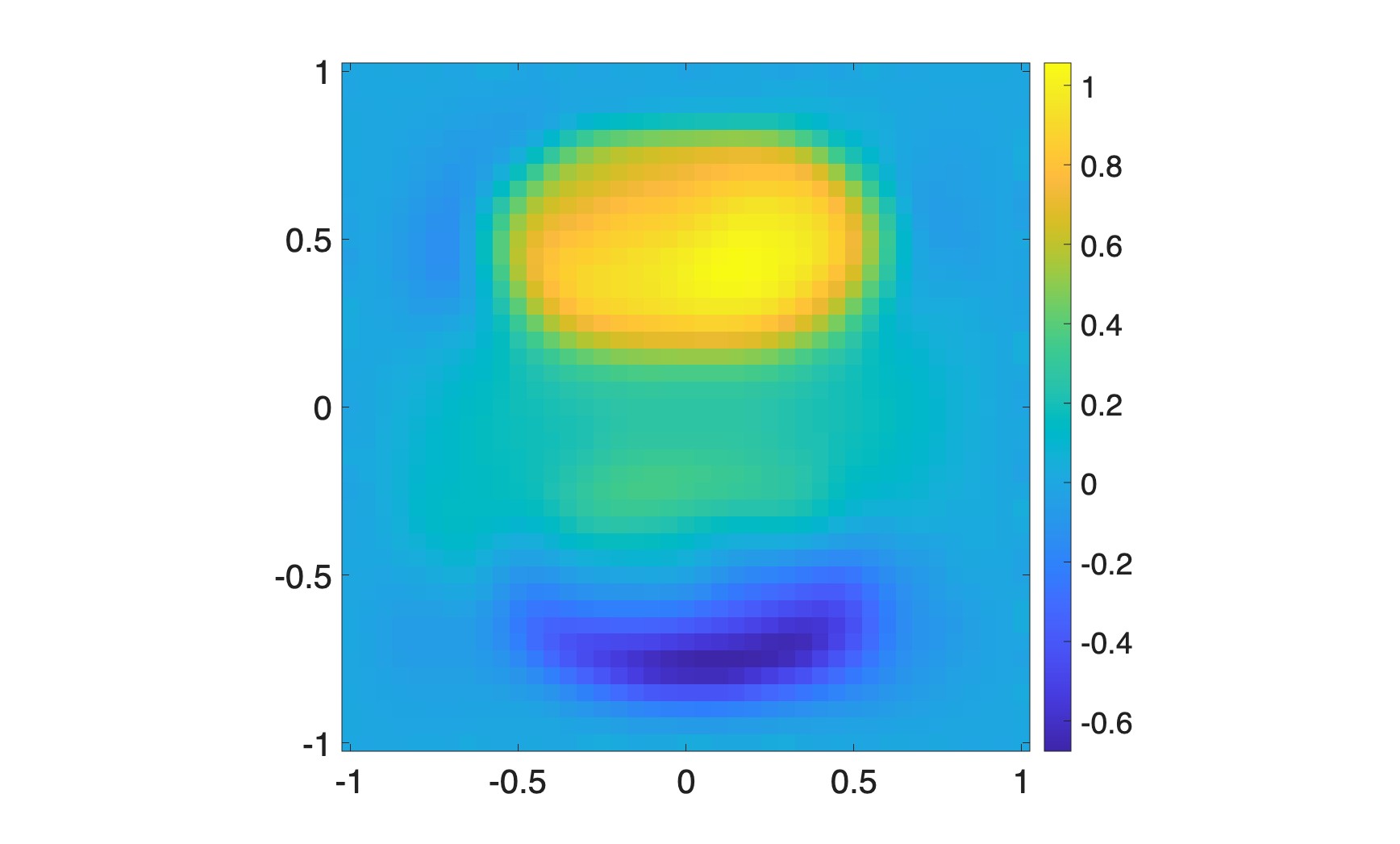}
}
\hfill
\subfloat[$ p_2^{\rm comp} $]{
\includegraphics[width=0.23\textwidth]{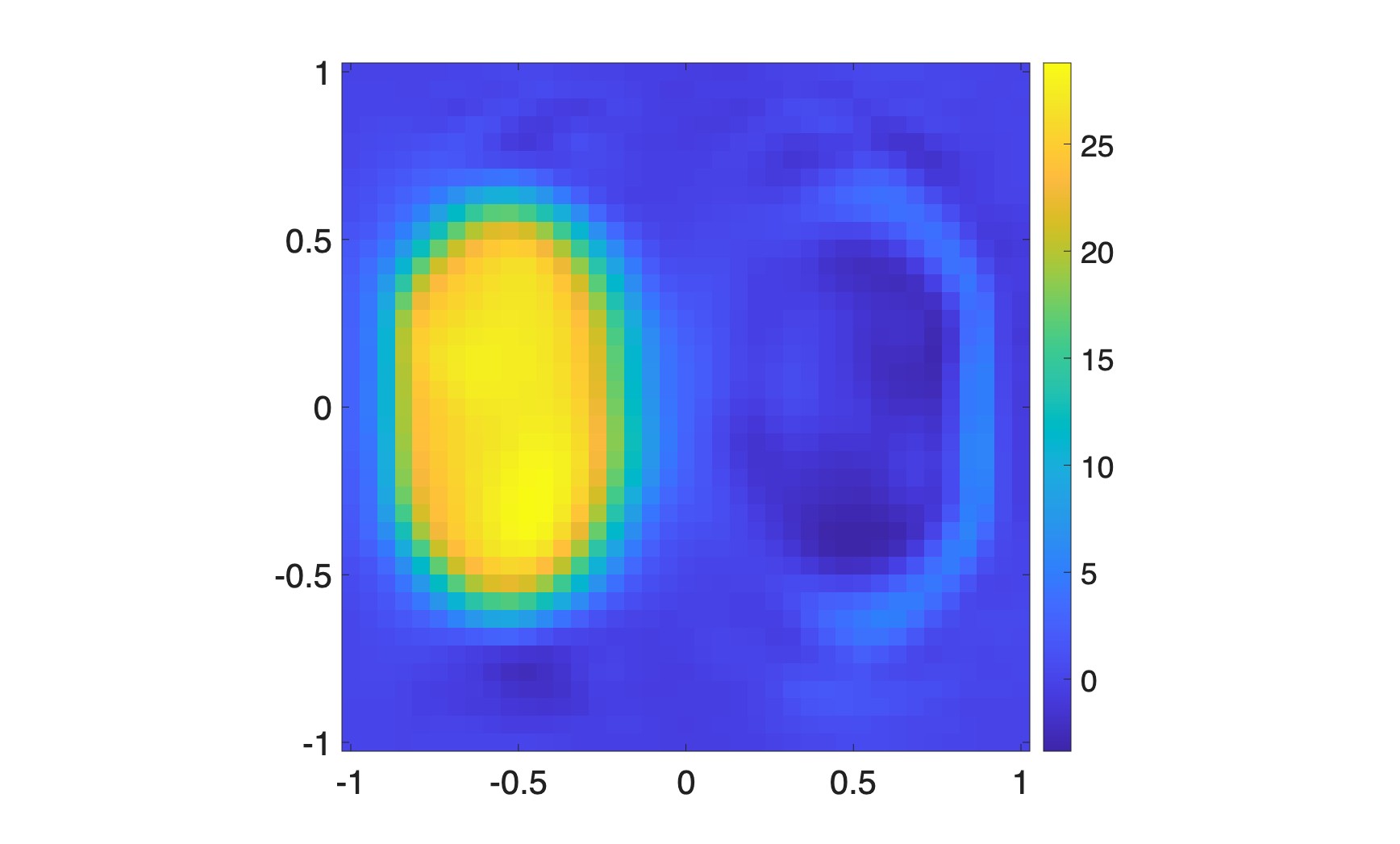}
}
\hfill
\subfloat[ $ p_2^{\rm comp} $]{
\includegraphics[width=0.23\textwidth]{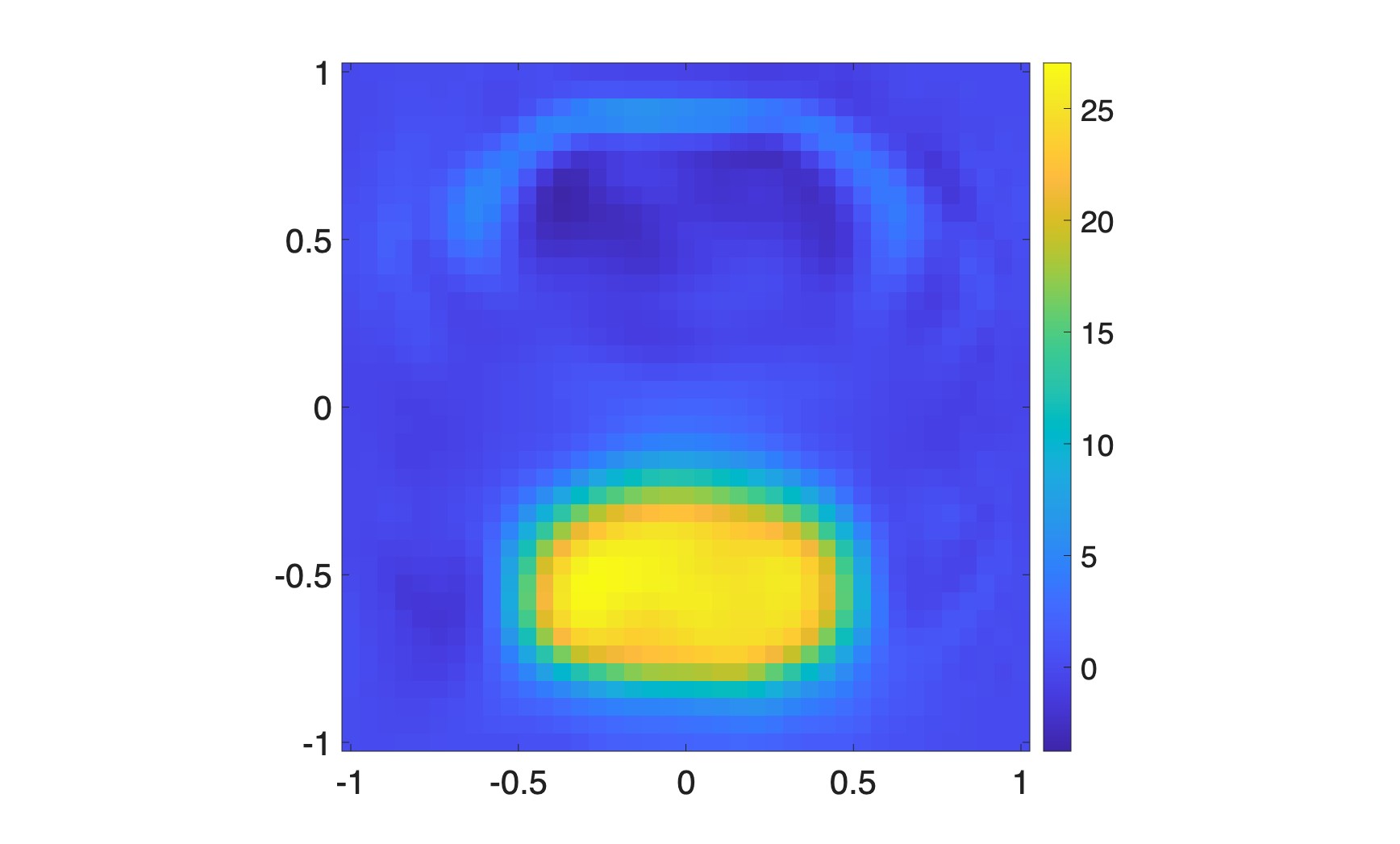}
}

\caption{\label{fig_test1} 
Comparison between the true and reconstructed initial displacement and velocity components for Test  1. The top row (a)--(d) shows the true initial data: displacement components $p_1^{\rm true}$ and $p_2^{\rm true}$, and velocity components $q_1^{\rm true}$ and $q_2^{\rm true}$. The bottom row (e)--(h) displays the corresponding reconstructed components $p_1^{\rm comp}$, $p_2^{\rm comp}$, $q_1^{\rm comp}$, and $q_2^{\rm comp}$ obtained by the proposed time-reduction method.
}
\end{figure}

In this test, each component of the initial displacement and velocity fields is modeled by a single elliptical inclusion located asymmetrically within the domain to ensure spatial distinguishability between components.
The computational results in Figure~\ref{fig_test1} demonstrate the effectiveness of the proposed time-reduction method in recovering the initial displacement and velocity fields. The reconstructed functions $p_1^{\rm comp}$, $p_2^{\rm comp}$, $q_1^{\rm comp}$, and $q_2^{\rm comp}$ show good agreement with their true counterparts. In particular, the shape, size, and location of the elliptical inclusions are well captured, indicating that the method successfully detects the key structural features of the initial conditions. 

However, some artifacts are visible in the reconstructions, especially in the velocity components. These distortions are likely due to the coupled nature of the elastic wave equation, where interactions between compressional and shear waves introduce complexities that can challenge the inversion process, particularly when only partial boundary data are used. Nevertheless, the reconstructions provide a qualitatively accurate approximation of the ground truth and affirm the reliability of the proposed numerical scheme.

In addition to successfully recovering the shape and location of the inclusions, the numerical reconstruction also provides a reasonable approximation of the amplitude of the displacement and velocity fields. Specifically, the reconstructed maximum value of $p_1^{\rm comp}$ is 1.0537, corresponding to a relative error of 5.37\%. For $p_2^{\rm comp}$, the maximum is  1.0563 with a relative error of 5.63\%. The maximum of $q_1^{\rm comp}$ is 28.7903, yielding a relative error of 4.03\%, while $q_2^{\rm comp}$ reaches 27.0403 with a relative error of 9.87\%. These quantitative results indicate that the time-dimensional reduction method achieves satisfactory accuracy, particularly for the displacement components, despite the inverse nature and coupling effects inherent in the elastic system.

\subsubsection*{Test 2} 
We conduct the numerical experiment in the setting of a nonhomogeneous isotropic medium, where the elasticity tensor $ \mathbb{C} = (C_{ijkl}(x, y)) $ varies spatially and is defined as
\begin{equation}
C_{ijkl}(x, y) = \lambda(x, y) \, \delta_{ij} \delta_{kl} + \mu(x, y) \left( \delta_{ik} \delta_{jl} + \delta_{il} \delta_{kj} \right),
\end{equation}
for all $ (x, y) \in \mathbb{R}^2 $. The spatially dependent Lam\'e parameters are given by
\[
\mu(x, y) = 2 + \sin(xy), 
\quad \text{and} \quad 
\lambda(x, y) = 1 + e^{-0.5(x^2 + y^2)}.
\]
This formulation reflects an isotropic material with smoothly varying stiffness across the domain.

In this test case, the true initial displacement $\mathbf{p}^{\rm true} = (p_1^{\rm true}, p_2^{\rm true})$ and velocity $\mathbf{q}^{\rm true} = (q_1^{\rm true}, q_2^{\rm true})$ are defined as follows:
\begin{align*}
p_1^{\rm true}(x, y) &=
\begin{cases}
2, & \text{if } 3(x + 0.3)^2 + 15(y - 0.3)^2 < 0.8^2, \\
0, & \text{otherwise},
\end{cases}
\\
p_2^{\rm true}(x, y) &=
\begin{cases}
2, & \text{if } 18(x - 0.3)^2 + 4y^2 < 1, \\
0, & \text{otherwise}
\end{cases}
\end{align*}
and
\begin{align*}
q_1^{\rm true}(x, y) &=
\begin{cases}
20, & \text{if } \max(18(x - 0.5)^2, 4y^2) < 1, \\
0, & \text{otherwise},
\end{cases}
\\
q_2^{\rm true}(x, y) &=
\begin{cases}
20, & \text{if } \max(15(x + y + 0.3)^2, 4y^2) < 1, \\
0, & \text{otherwise}.
\end{cases}
\end{align*}

\begin{figure}[h!]
\centering
\subfloat[ $ p_1^{\rm true} $]{
\includegraphics[width=0.23\textwidth]{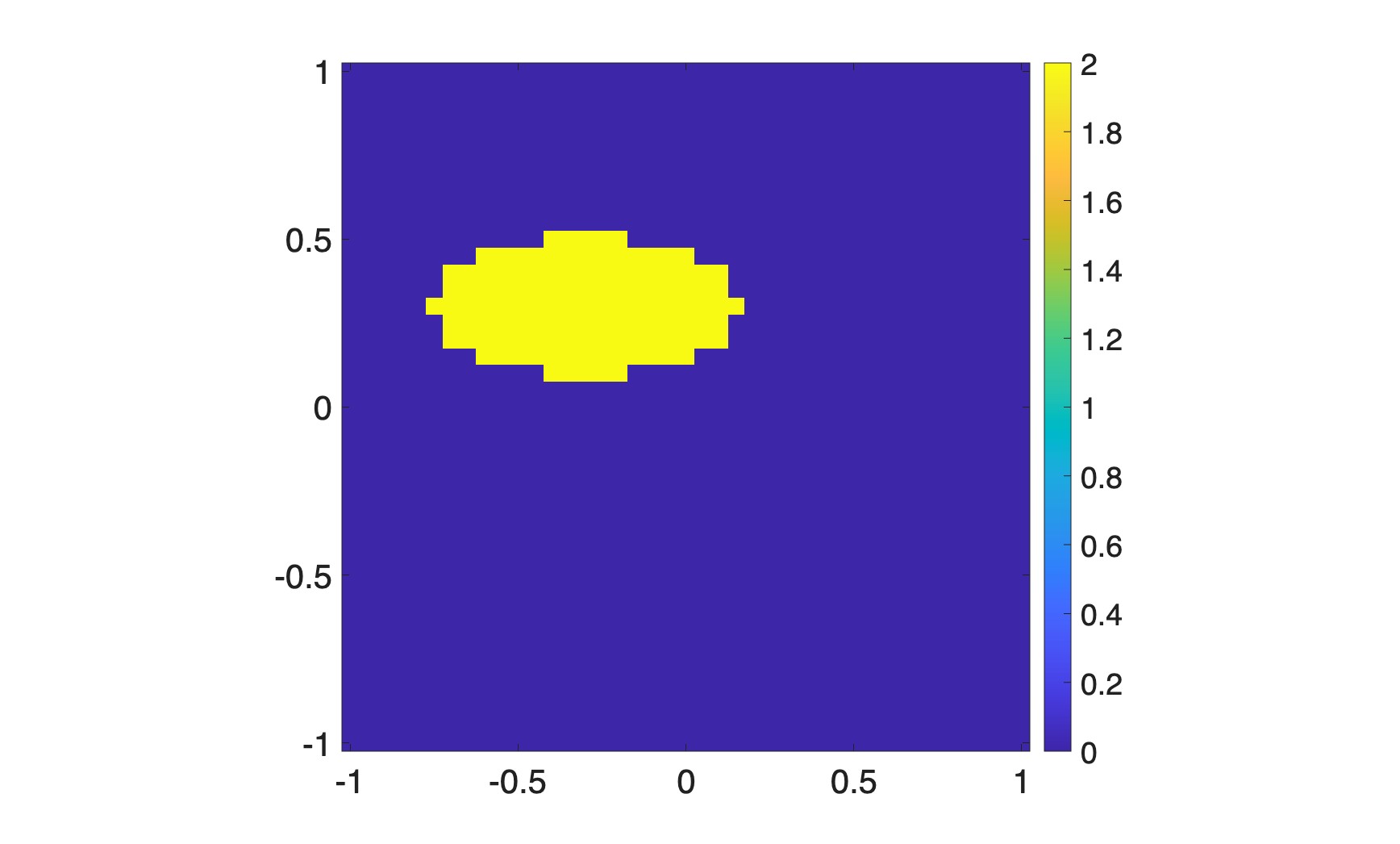}
}
\hfill
\subfloat[ $ p_2^{\rm true} $]{
\includegraphics[width=0.23\textwidth]{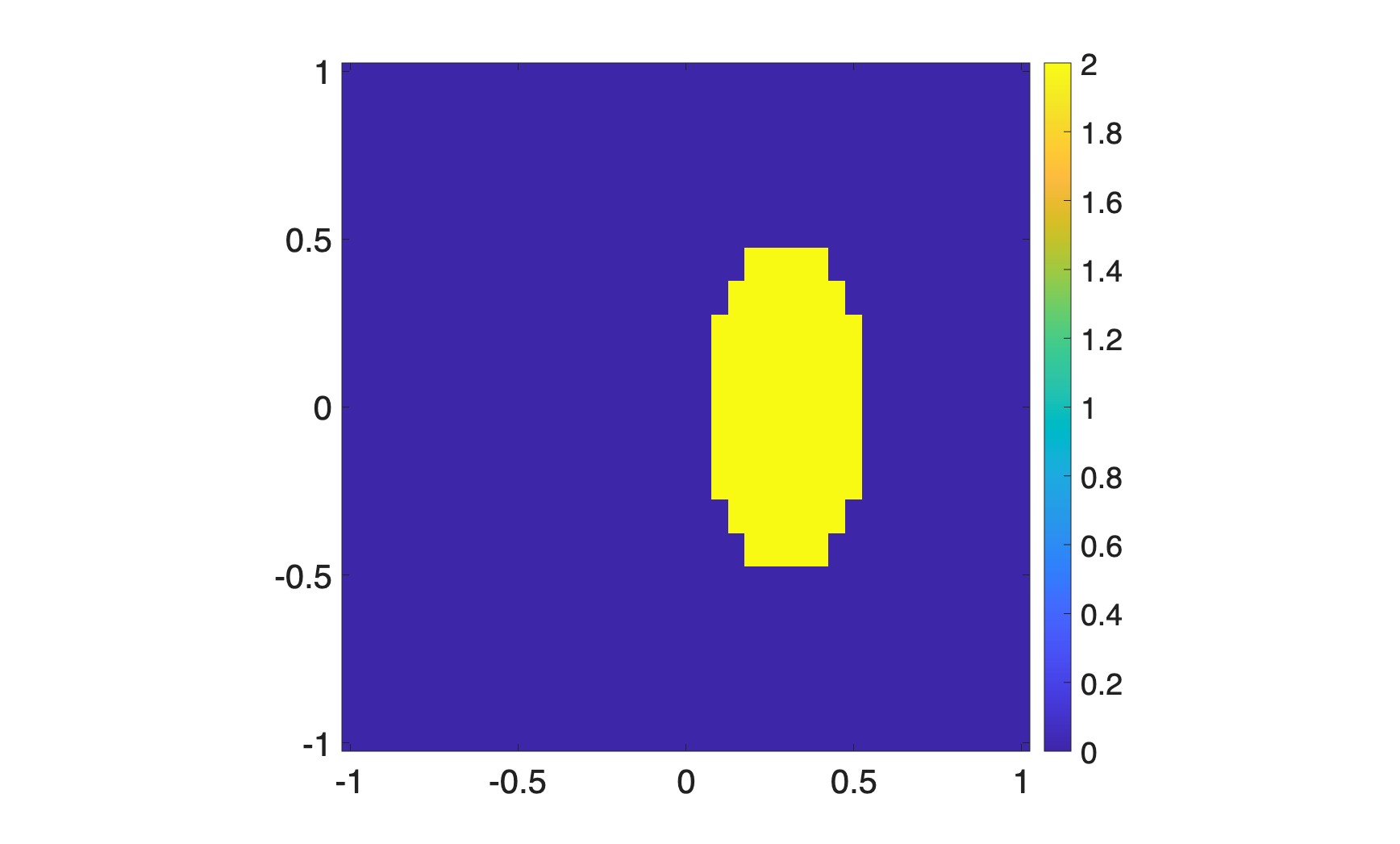}
}
\hfill \subfloat[ $ q_1^{\rm true} $]{
\includegraphics[width=0.23\textwidth]{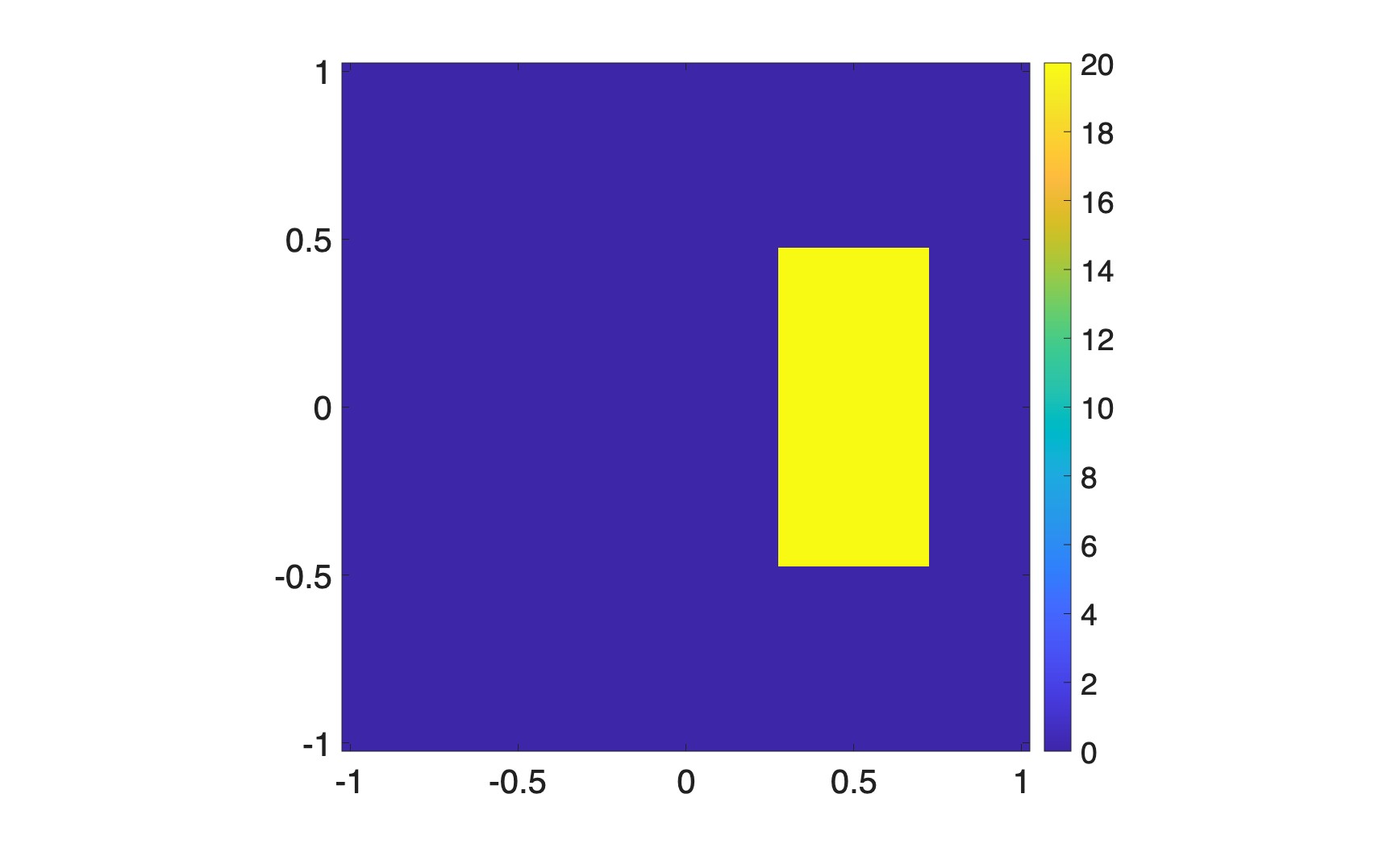}
}
\hfill \subfloat[ $ q_2^{\rm true} $]{
\includegraphics[width=0.23\textwidth]{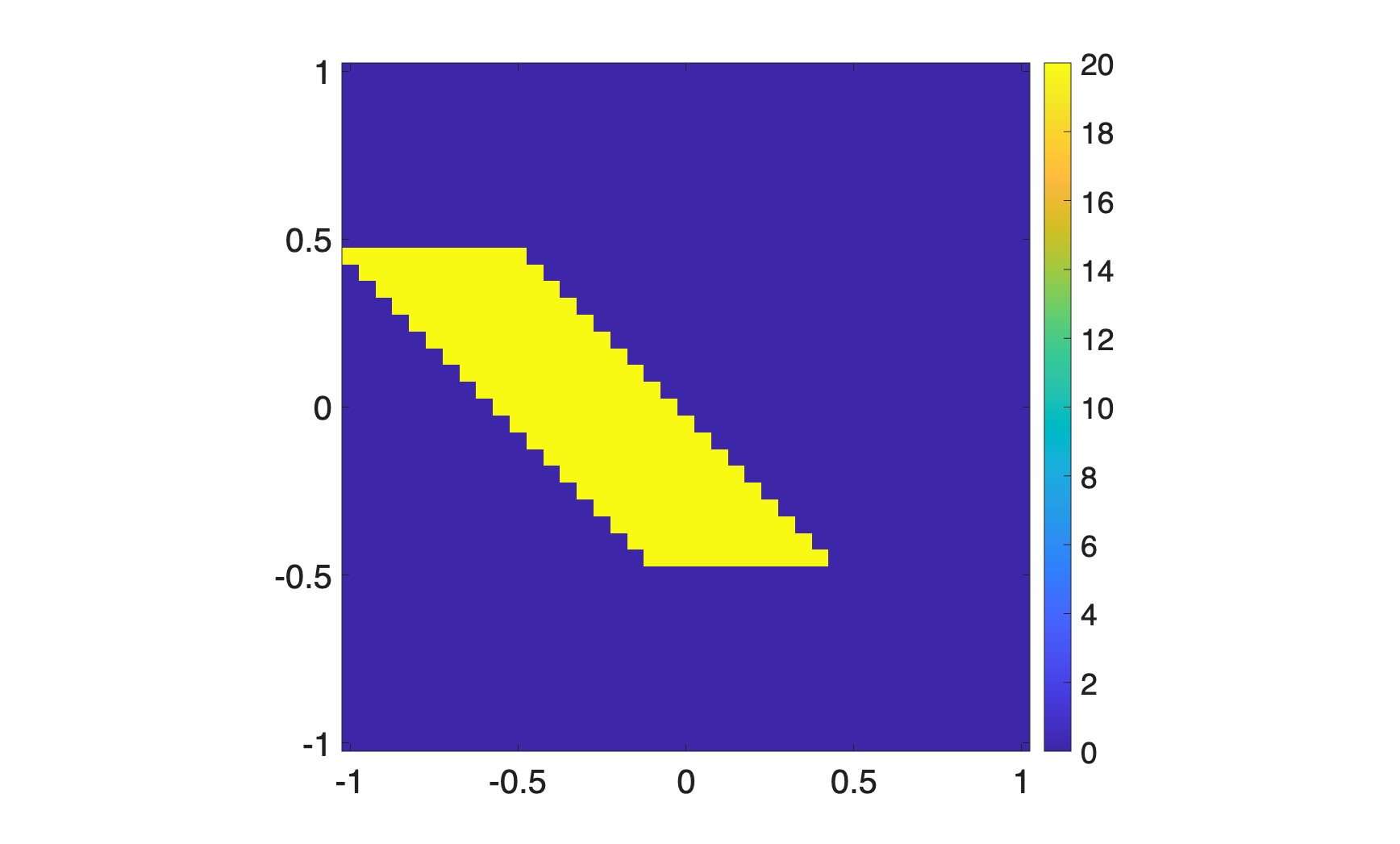}
}

\subfloat[ $ p_1^{\rm comp} $]{
\includegraphics[width=0.23\textwidth]{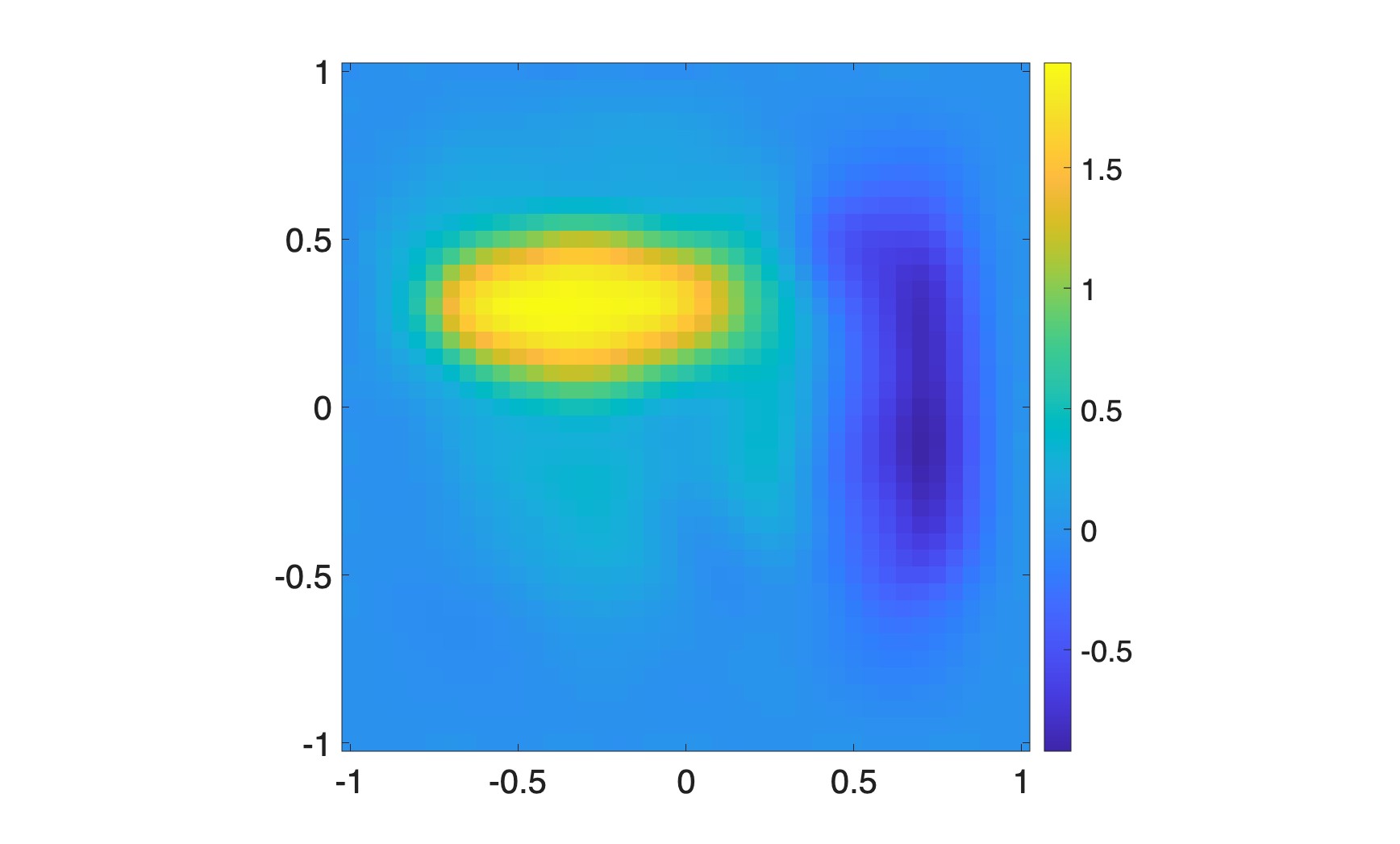}
}
\hfill
\subfloat[ $ p_2^{\rm comp} $]{
\includegraphics[width=0.23\textwidth]{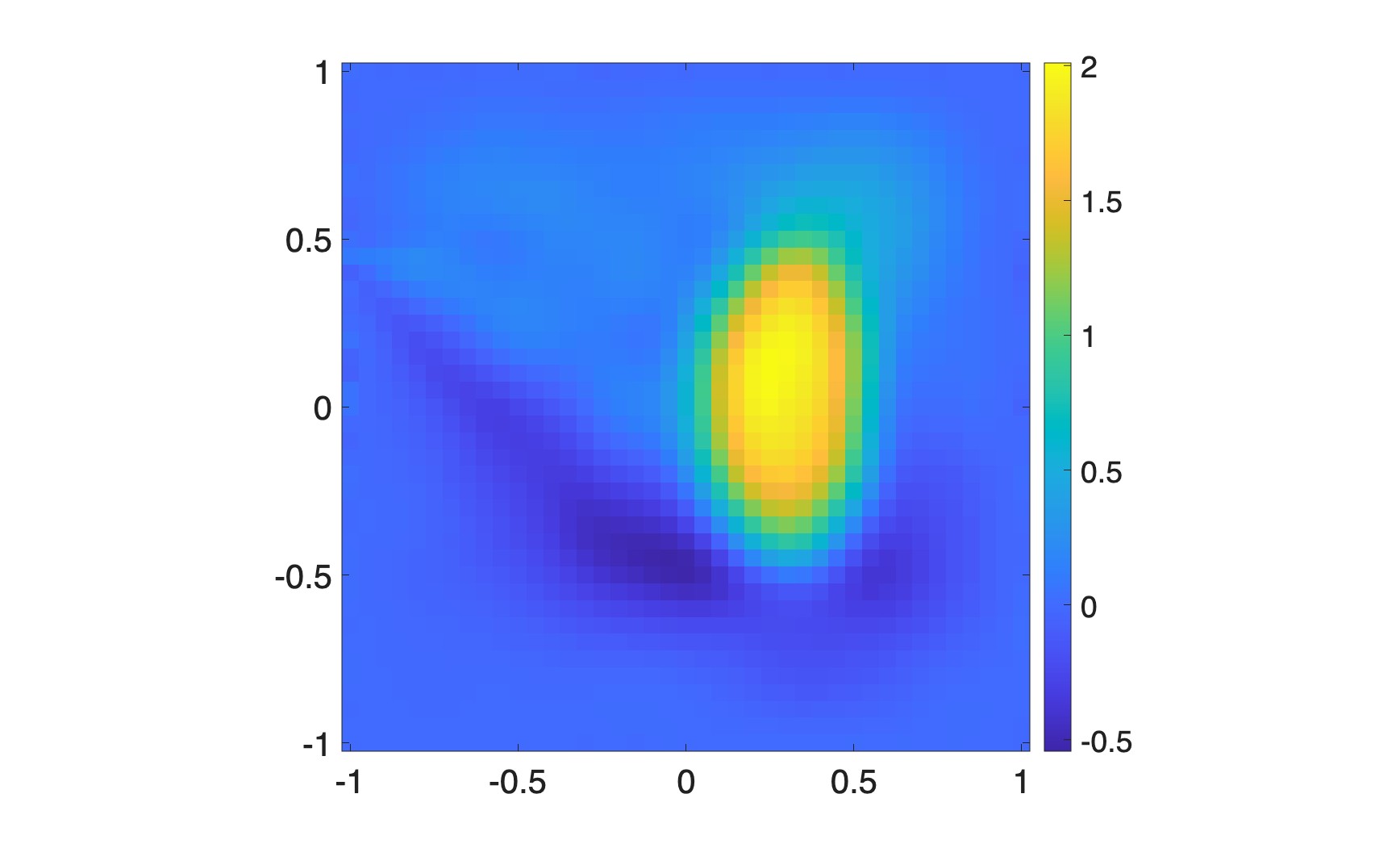}
}
\hfill
\subfloat[ $ p_2^{\rm comp} $]{
\includegraphics[width=0.23\textwidth]{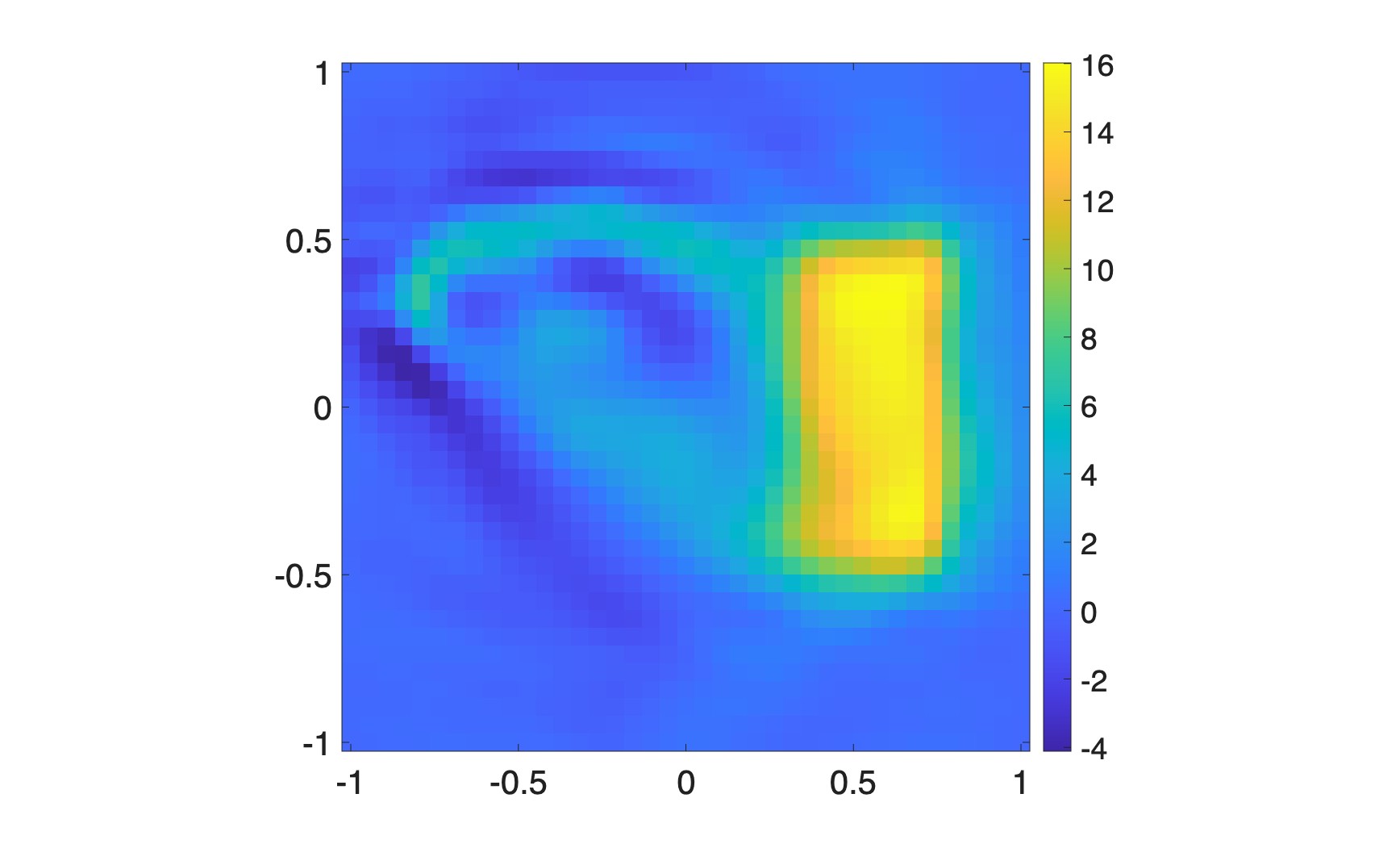}
}
\hfill
\subfloat[ $ p_2^{\rm comp} $]{
\includegraphics[width=0.23\textwidth]{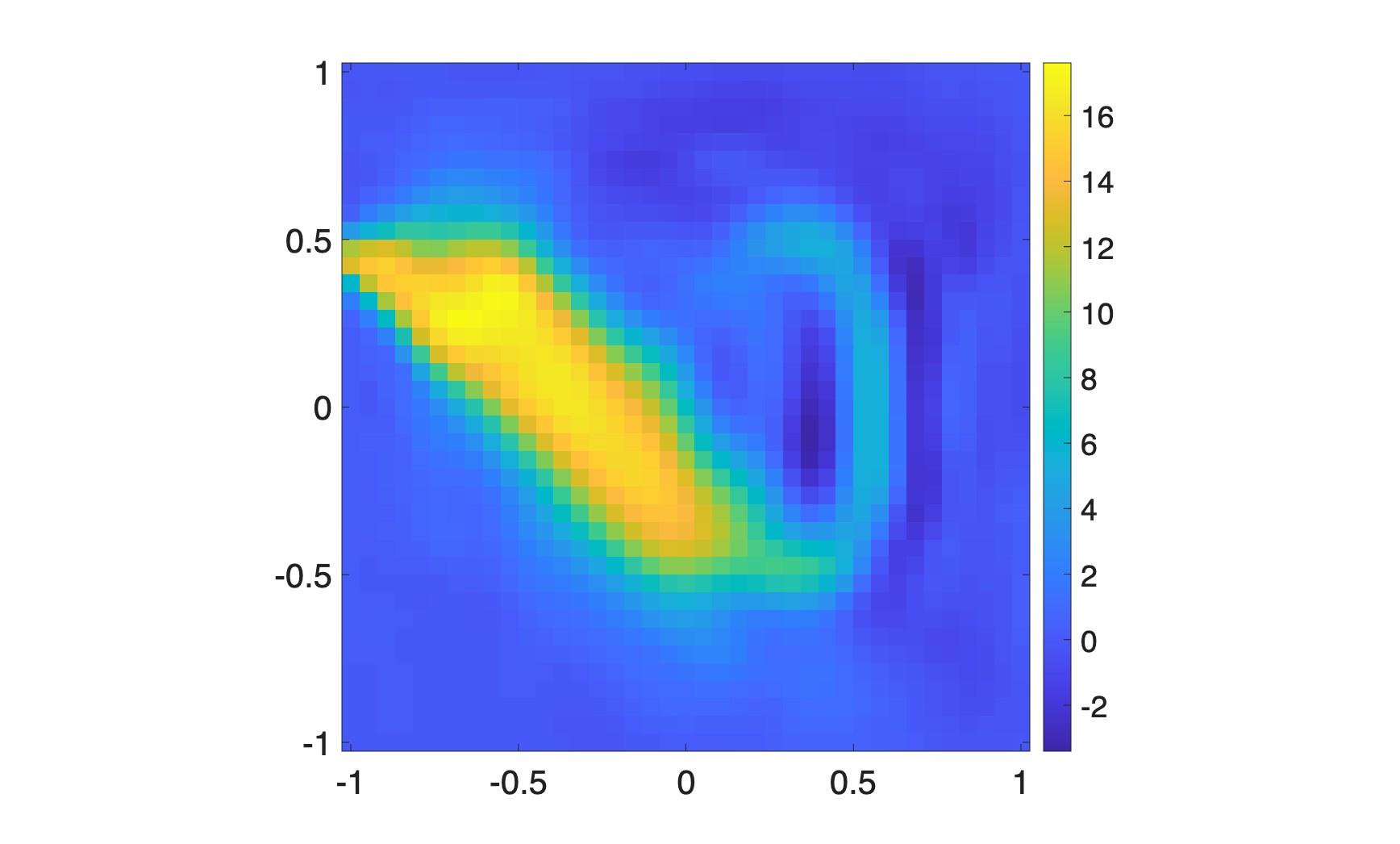}
}

\caption{\label{fig_test2} 
Comparison between the true and reconstructed initial displacement and velocity components for Test  2. The top row (a)--(d) shows the true initial data: displacement components $p_1^{\rm true}$ and $p_2^{\rm true}$, and velocity components $q_1^{\rm true}$ and $q_2^{\rm true}$. The bottom row (e)--(h) displays the corresponding reconstructed components $p_1^{\rm comp}$, $p_2^{\rm comp}$, $q_1^{\rm comp}$, and $q_2^{\rm comp}$ obtained by the proposed time-reduction method.
}
\end{figure}

Figure~\ref{fig_test2} presents the true and reconstructed initial displacement and velocity components for the second test case. The reconstructions successfully capture the general shape, size, and orientation of the inclusions in all four components. In particular, the rectangular features in $q_1^{\rm true}$ and the slanted structure in $q_2^{\rm true}$ are identifiable in their respective reconstructions $q_1^{\rm comp}$ and $q_2^{\rm comp}$. Similarly, the elliptical structures in $p_1^{\rm true}$ and $p_2^{\rm true}$ are accurately localized and well approximated in $p_1^{\rm comp}$ and $p_2^{\rm comp}$.

Despite these strengths, some artifacts and oscillatory patterns are visible in the reconstructions, especially in $q_2^{\rm comp}$ and $p_2^{\rm comp}$, where spurious background fluctuations appear. These artifacts may stem from the coupled nature of the elastic system, which can induce interference between wave modes, and from the use of truncated spectral representations. Additionally, the sharp edges of the true inclusions are somewhat smoothed out in the computed results, suggesting that further regularization tuning or higher-frequency information could improve the spatial resolution. Overall, the reconstructions remain robust and informative, clearly demonstrating the potential of the proposed time-reduction method.

From a quantitative perspective, the numerical reconstruction achieves a reasonable approximation of the amplitude of the displacement and velocity fields. Specifically, the reconstructed maximum of $p_1^{\rm comp}$ is 1.9344, resulting in a relative error of 3.28\%. For $p_2^{\rm comp}$, the maximum is 2.0100 with a relative error of 0.5\%. Regarding the velocity components, the reconstructed maximum of $q_1^{\rm comp}$ is 16.0144, yielding a relative error of 19.93\%, while $q_2^{\rm comp}$ attains 17.607 with a relative error of 11.96\%. Considering that the input data includes $10\%$ additive noise, these quantitative reconstructions, especially for the displacement field, can be regarded as satisfactory. The moderate increase in error for the velocity components is expected due to their higher sensitivity and the inherent challenges posed by the coupled nature of the elastic wave system.

\subsubsection{Test 3}

We now consider the case in which $ \mathbb{C} $ represents an anisotropic elastic medium. To facilitate numerical implementation and matrix-based computations, we flatten the fourth-order elasticity tensor into a $ 4 \times 4 $ matrix by identifying the paired indices as follows: $ 11 \mapsto 1 $, $ 12 \mapsto 2 $, $ 21 \mapsto 3 $, and $ 22 \mapsto 4 $. Under this index mapping, the components of the tensor $ \mathbb{C} $ are organized into the following matrix form:
\[
\mathbb{C}_{\text{flat}} =
\begin{bmatrix}
C_{1111} & C_{1112} & C_{1121} & C_{1122} \\
C_{1211} & C_{1212} & C_{1221} & C_{1222} \\
C_{2111} & C_{2112} & C_{2121} & C_{2122} \\
C_{2211} & C_{2212} & C_{2221} & C_{2222}
\end{bmatrix}
=
\begin{bmatrix}
80 & 5 & 5 & 30 \\
5 & 20 & 20 & 0 \\
5 & 20 & 20 & 0 \\
30 & 0 & 0 & 40
\end{bmatrix}.
\]
This representation captures the anisotropic nature of the medium, where the elastic response depends on the direction of the applied deformation. It is readily verified that there do not exist values of $ \lambda $ and $ \mu $ such that $ \mathbb{C}_{\text{flat}} $ can be written in the isotropic form
\[
\mathbb{C}_{\text{iso}} =
\begin{bmatrix}
\lambda + 2\mu & 0 & 0 & \lambda \\
0 & \mu & 0 & 0 \\
0 & 0 & \mu & 0 \\
\lambda & 0 & 0 & \lambda + 2\mu
\end{bmatrix},
\]
confirming that the medium modeled by $ \mathbb{C} $ is indeed anisotropic.

In Test 3, the true initial displacement and velocity fields $\mathbf{p}^{\rm true} = (p_1^{\rm true}, p_2^{\rm true})$ and $\mathbf{q}^{\rm true} = (q_1^{\rm true}, q_2^{\rm true})$ are constructed using a combination of elliptical and composite-shaped inclusions. Specifically, the first displacement component is defined as
\[
p_1^{\rm true}(x, y) = 
\begin{cases}
1 & \text{if } 8x^2 + 2y^2 < 0.5^2, \\
0 & \text{otherwise},
\end{cases}
\]
representing a horizontally stretched ellipse centered at the origin. The second displacement component is the union of two inclusions:
\[
p_2^{\rm true}(x, y) = 
\begin{cases}
1 & \text{if } 2x^2 + 13(y + 0.6)^2 < 0.7^2 \text{ or } 2(x + 0.5)^2 < 0.3^2 \text{ and } (y - 0.5)^2 < 0.3^2, \\
0 & \text{otherwise}.
\end{cases}
\]

For the velocity field, $q_1^{\rm true}$ is defined as a ring-like inclusion centered at the origin:
\[
q_1^{\rm true}(x, y) = 
\begin{cases}
40 & \text{if } 0.5^2 < x^2 + y^2 < 0.8^2, \\
0 & \text{otherwise},
\end{cases}
\]
and $q_2^{\rm true}$ corresponds to a diagonal-shaped inclusion defined by
\[
q_2^{\rm true}(x, y) = 
\begin{cases}
40 & \text{if } \max((x + y)^2, 20(x - y)^2) < 1, \\
0 & \text{otherwise}.
\end{cases}
\]
These configurations pose significant challenges for the inverse reconstruction due to the inclusion complexity and sharp geometric features.

\begin{figure}[h!]
\centering
\subfloat[ $ p_1^{\rm true} $]{
\includegraphics[width=0.23\textwidth]{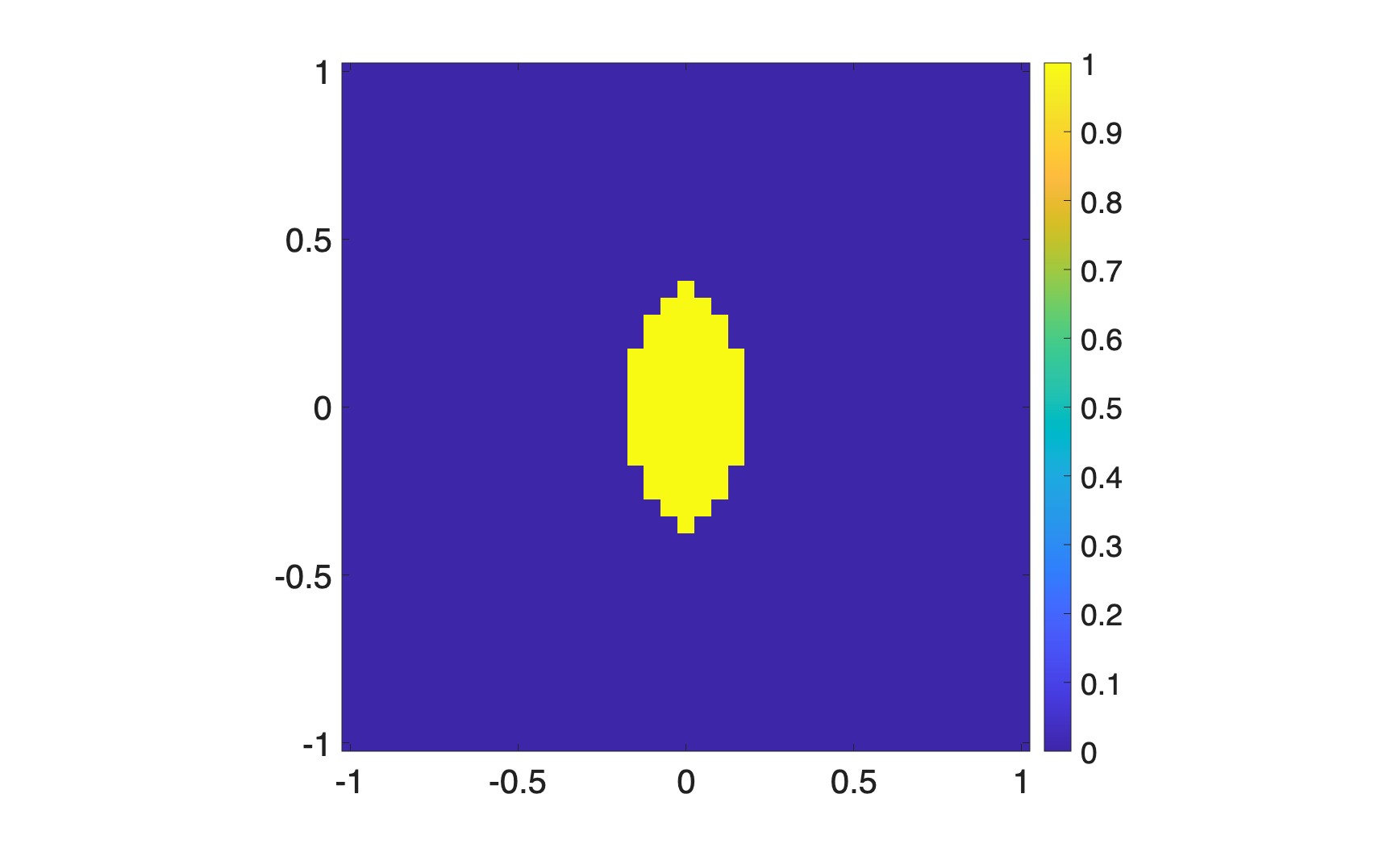}
}
\hfill
\subfloat[ $ p_2^{\rm true} $]{
\includegraphics[width=0.23\textwidth]{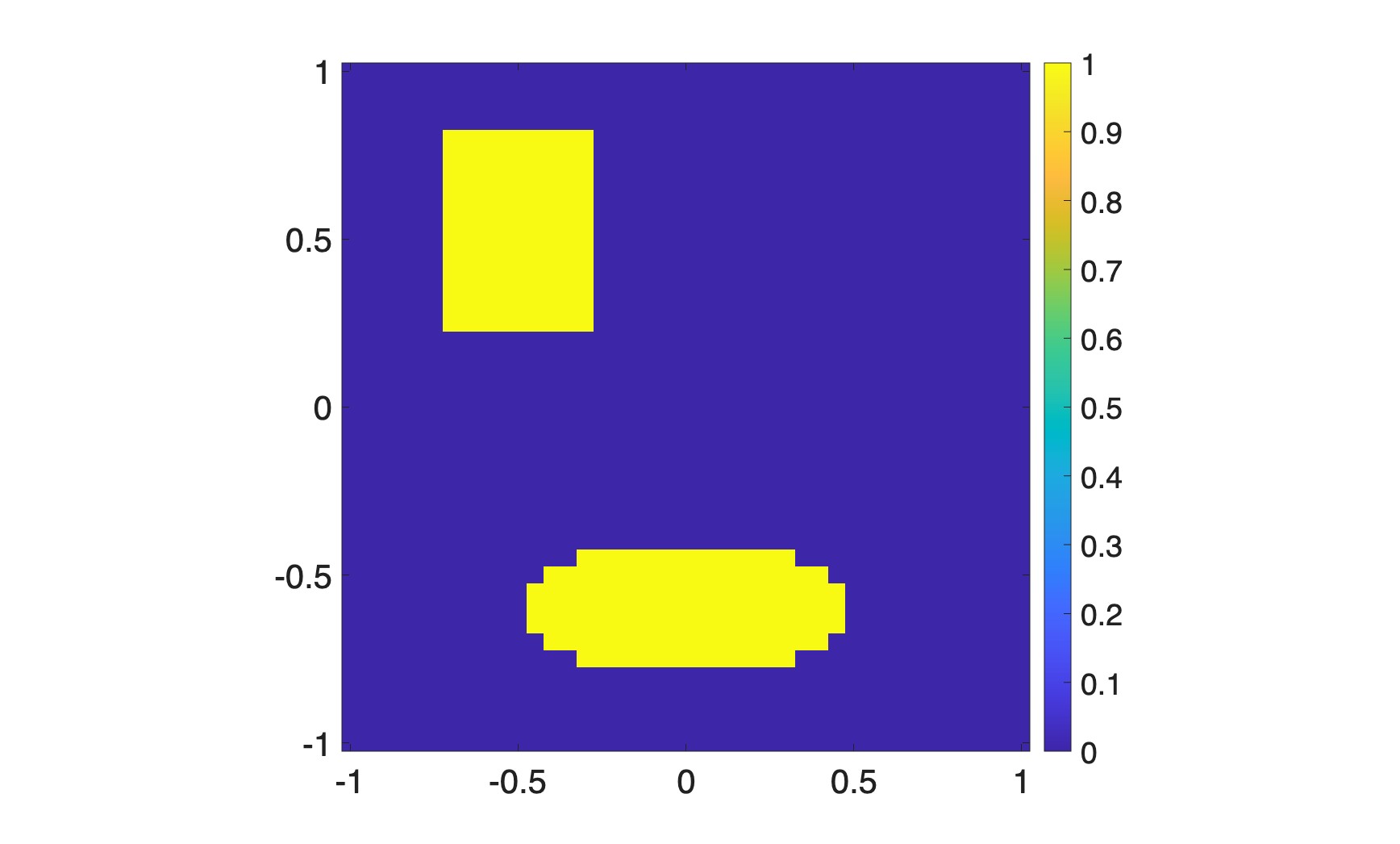}
}
\hfill \subfloat[ $ q_1^{\rm true} $]{
\includegraphics[width=0.23\textwidth]{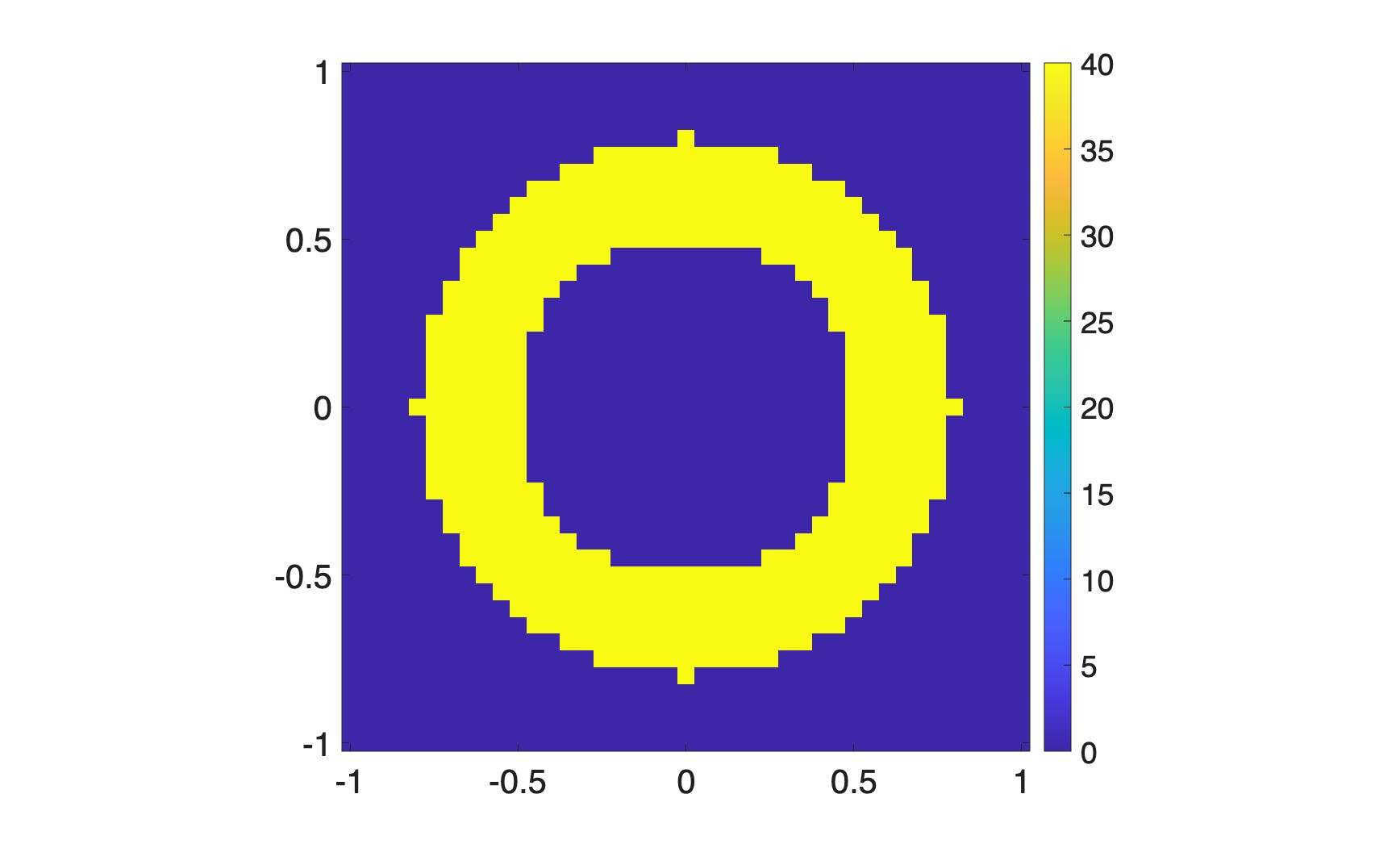}
}
\hfill \subfloat[ $ q_2^{\rm true} $]{
\includegraphics[width=0.23\textwidth]{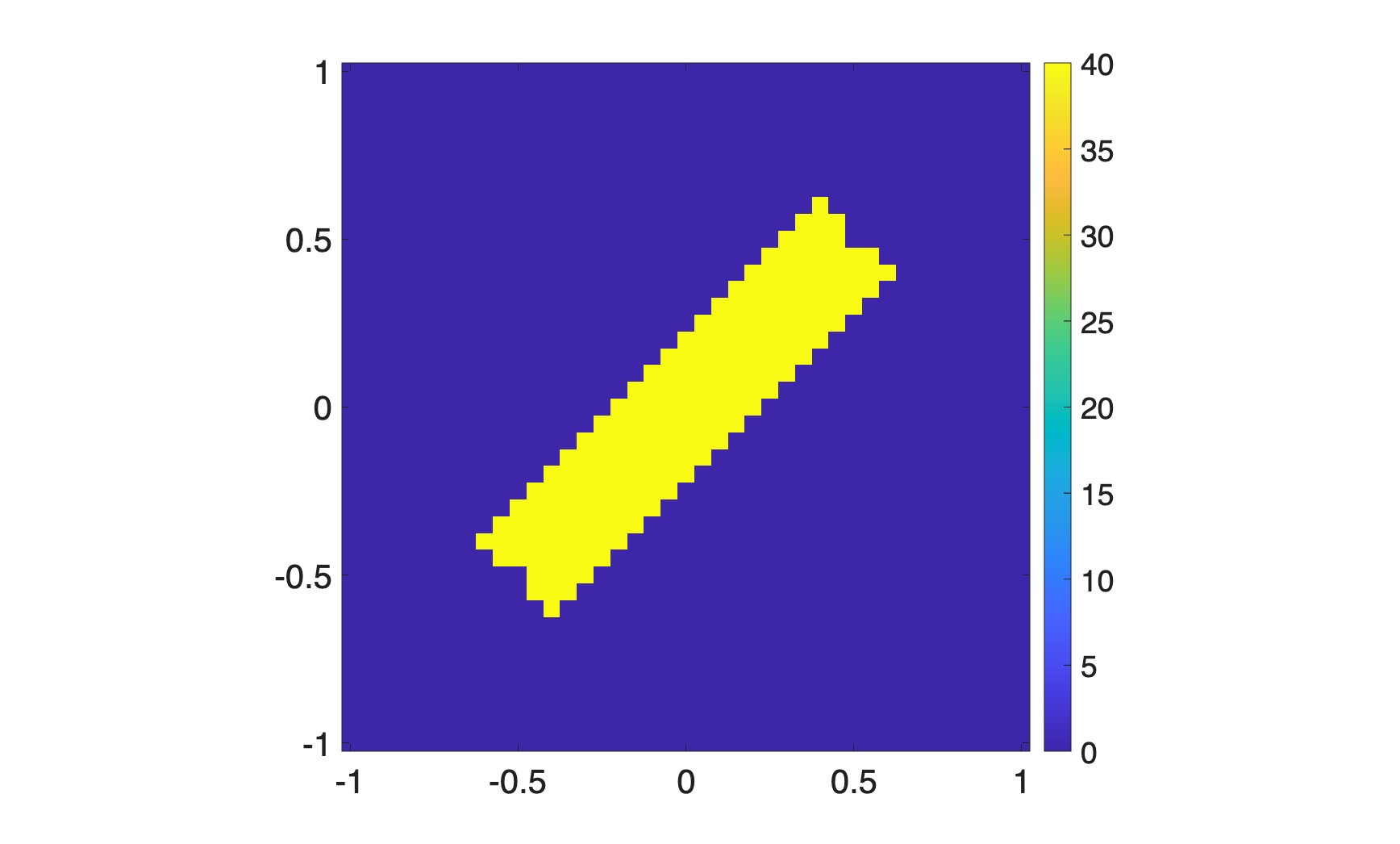}
}

\subfloat[ $ p_1^{\rm comp} $]{
\includegraphics[width=0.23\textwidth]{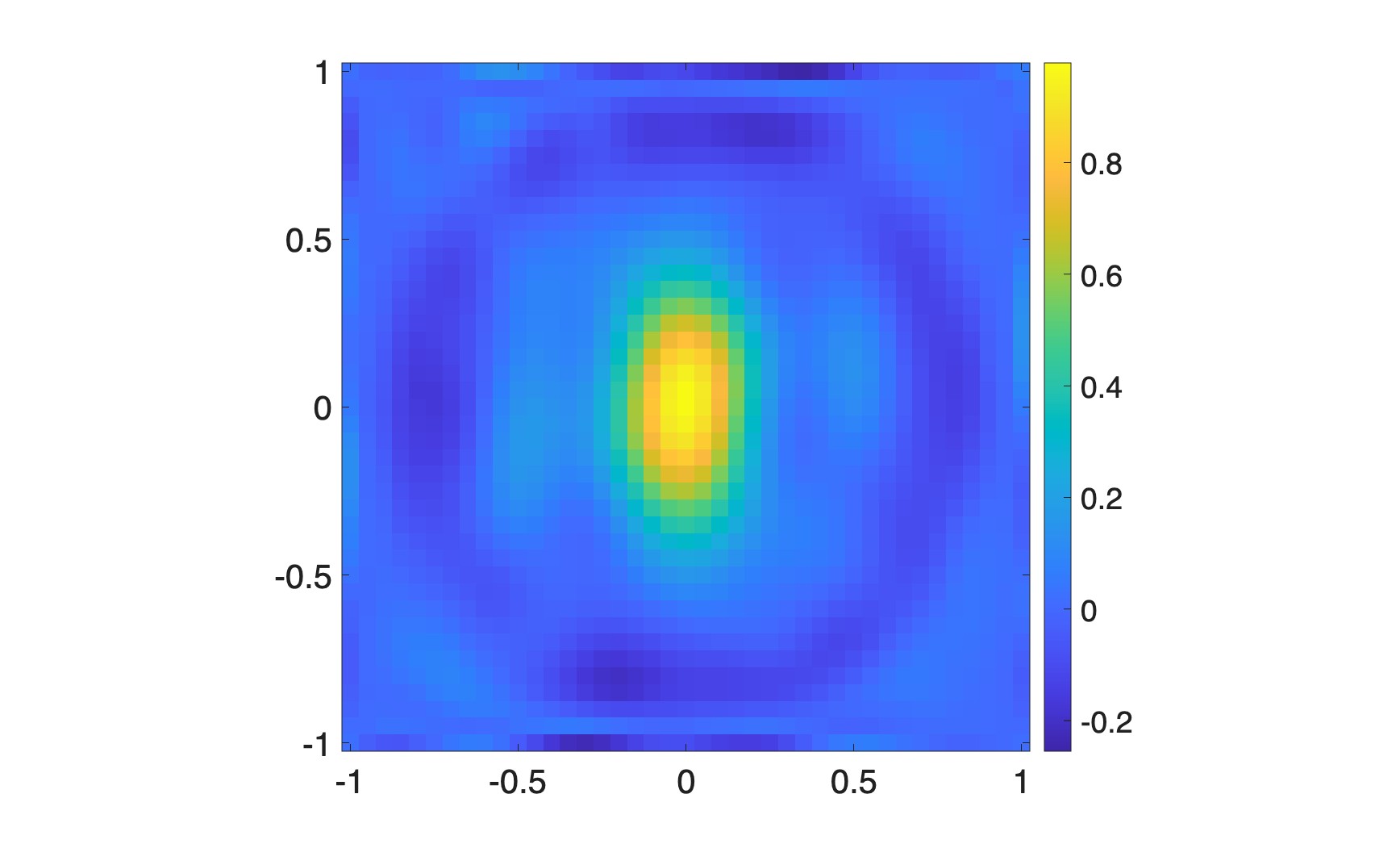}
}
\hfill
\subfloat[ $ p_2^{\rm comp} $]{
\includegraphics[width=0.23\textwidth]{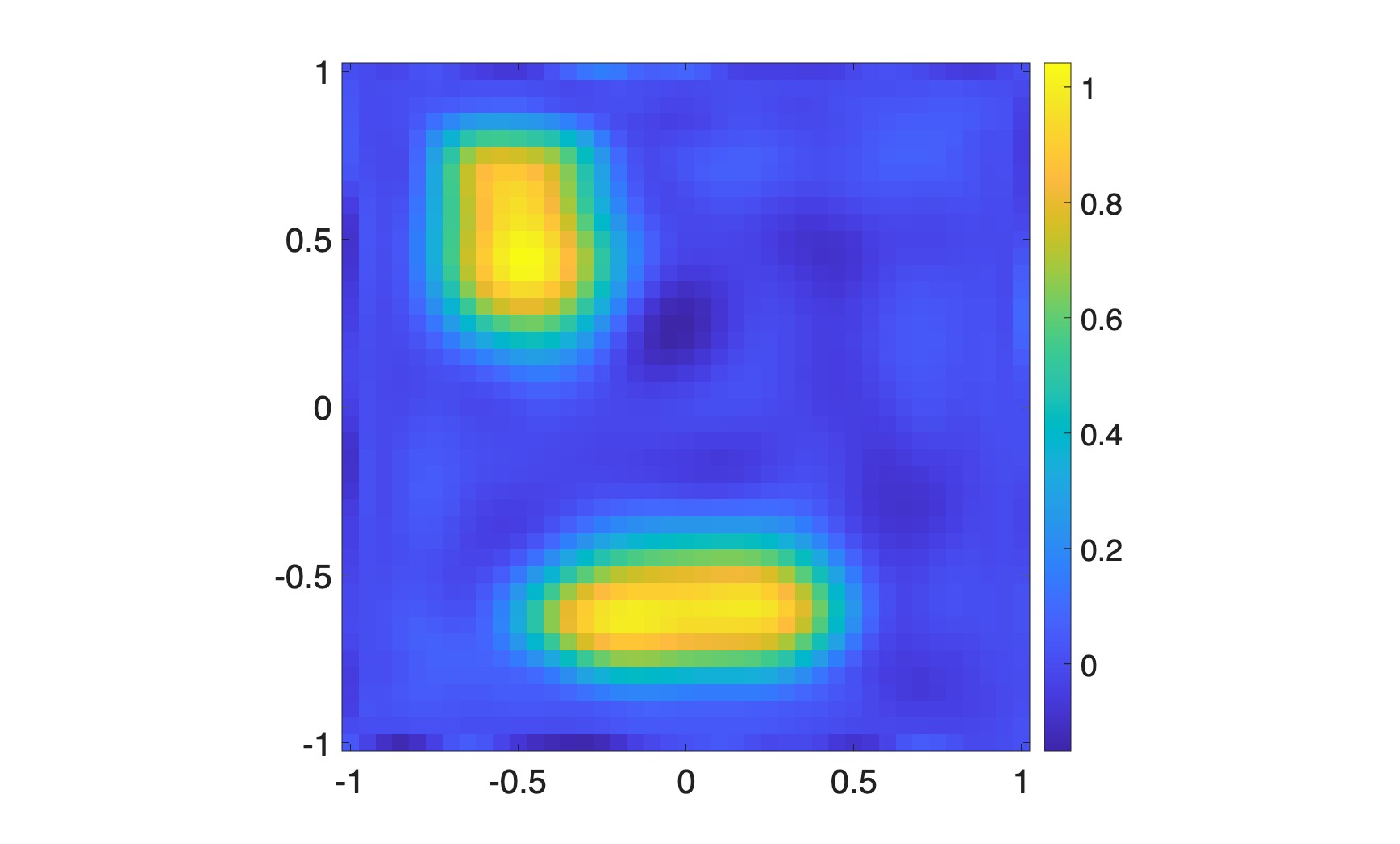}
}
\hfill
\subfloat[ $ p_2^{\rm comp} $]{
\includegraphics[width=0.23\textwidth]{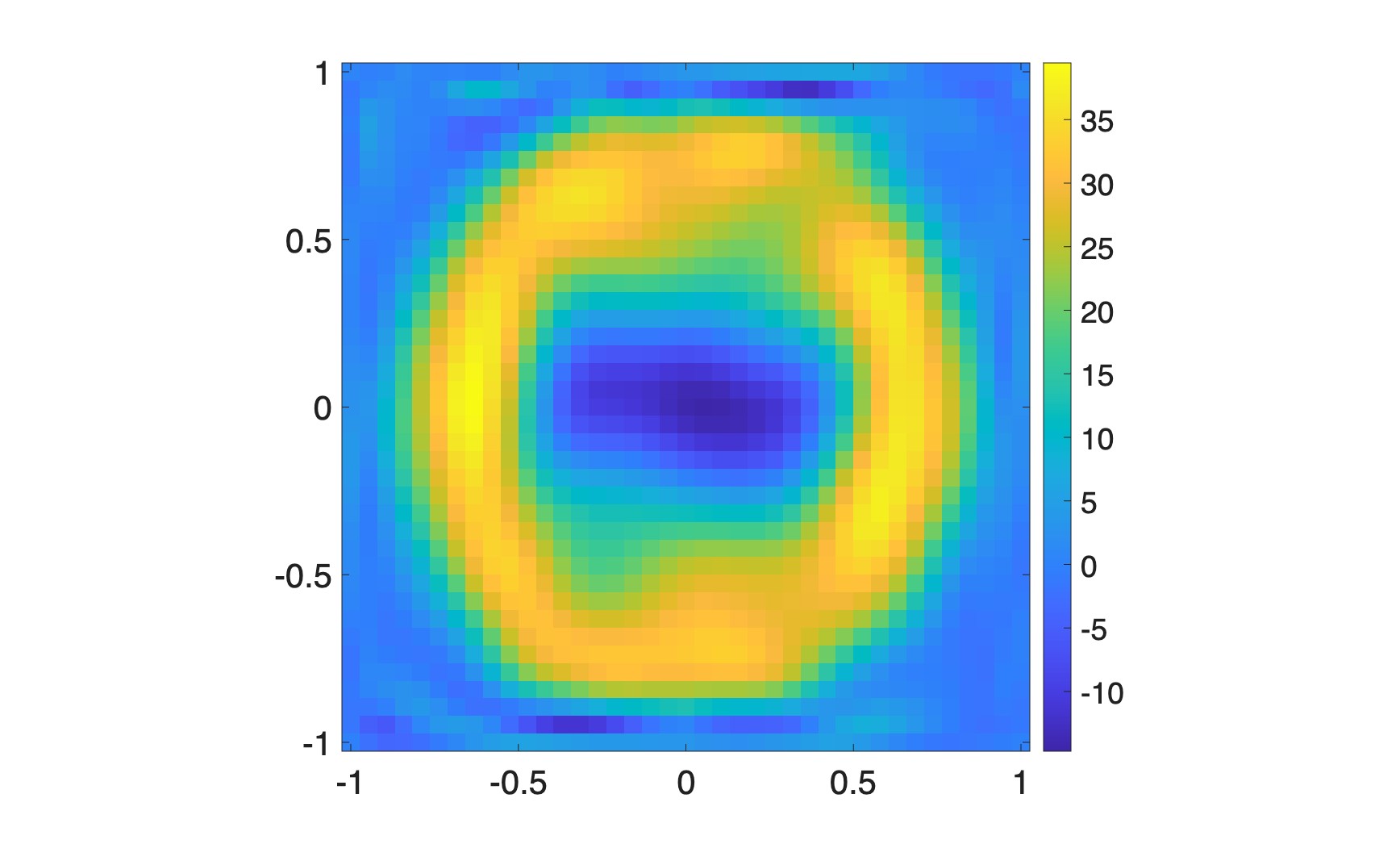}
}
\hfill
\subfloat[ $ p_2^{\rm comp} $]{
\includegraphics[width=0.23\textwidth]{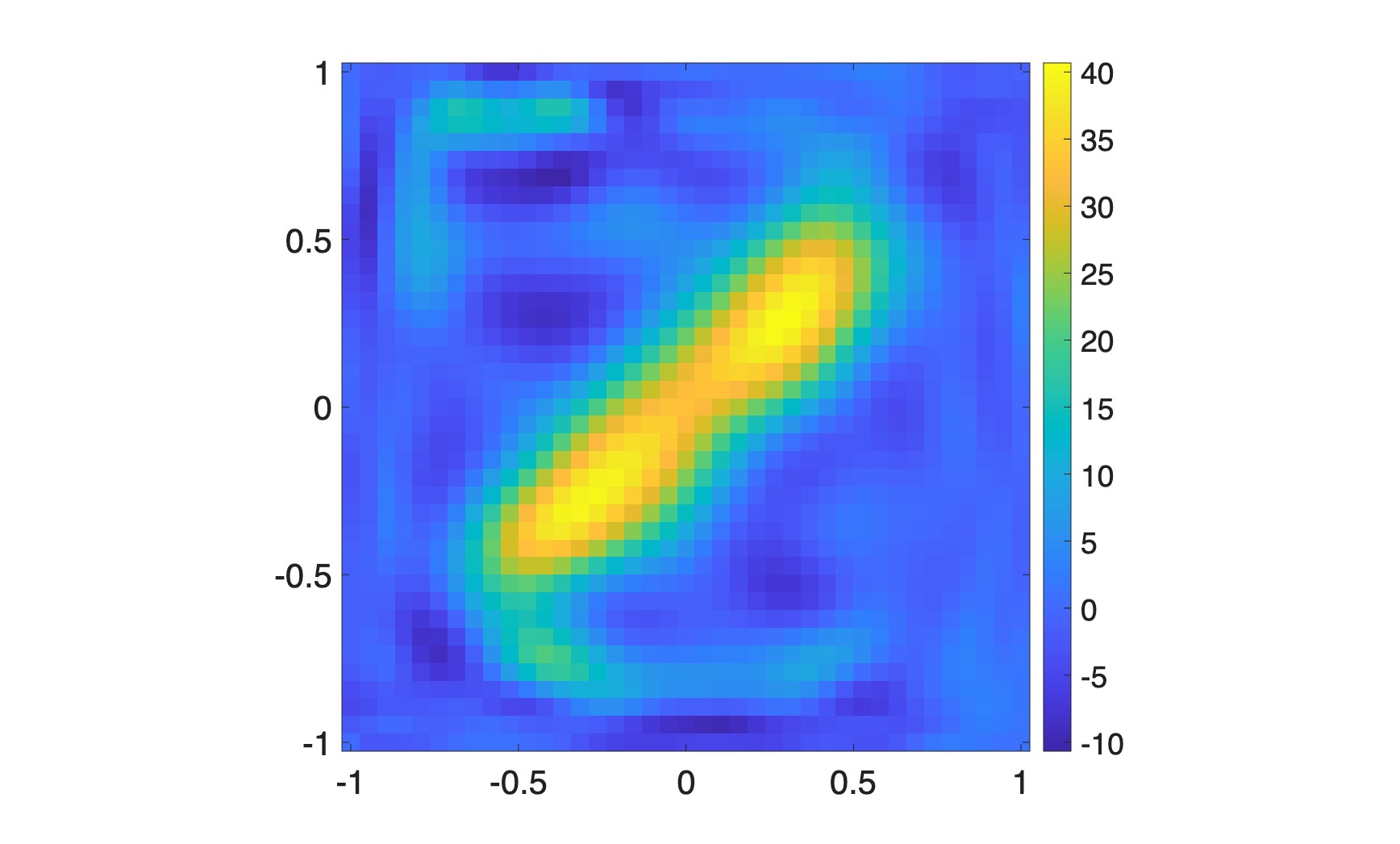}
}

\caption{\label{fig_test3} 
Comparison between the true and reconstructed initial displacement and velocity components for Test  3. The top row (a)--(d) shows the true initial data: displacement components $p_1^{\rm true}$ and $p_2^{\rm true}$, and velocity components $q_1^{\rm true}$ and $q_2^{\rm true}$. The bottom row (e)--(h) displays the corresponding reconstructed components $p_1^{\rm comp}$, $p_2^{\rm comp}$, $q_1^{\rm comp}$, and $q_2^{\rm comp}$ obtained by the proposed time-reduction method.
}
\end{figure}

The numerical results in Figure~\ref{fig_test3} illustrate the performance of the proposed time-dimensional reduction method in reconstructing the initial displacement and velocity fields. Overall, the reconstructed components capture the principal geometric features of the true data. The elliptical inclusion in \( p_1^{\text{true}} \) and the ring-like structure in \( q_1^{\text{true}} \) are identifiable in the computed solutions, showing that the method is capable of resolving curved and annular shapes effectively. The rectangular inclusion in \( p_2^{\text{true}} \), located in the upper-right part of the domain, is also reconstructed with correct orientation and position, though the edges are moderately blurred. Similarly, the slanted inclusion in \( q_2^{\text{true}} \) is approximately recovered, albeit with some smoothing and artifacts surrounding the interface.

Some challenges remain, particularly in the form of mild distortions near the inclusion boundaries and loss of sharpness in fine-scale features. These limitations may stem from the regularization process, the coupled structure of the elastic system, and the moderate noise level in the boundary data. Despite these issues, the method demonstrates strong robustness and fidelity, yielding quantitatively and qualitatively reliable approximations in complex multi-inclusion settings.

From a quantitative standpoint, the reconstructed displacement and velocity fields exhibit a satisfactory match with the true amplitudes. The computed maximum of $p_1^{\rm comp}$ is 0.9786, corresponding to a relative error of 2.14\%, while $p_2^{\rm comp}$ reaches 1.0418 with an error of 4.18\%. For the velocity components, the reconstructed maximum of $q_1^{\rm comp}$ is 39.4602 (1.35\% error), and that of $q_2^{\rm comp}$ is 40.6832 (1.71\% error). Given the presence of 10\% additive noise in the input data, these results, particularly for the displacement fields, demonstrate the robustness of the method. The slightly larger errors in the velocity components are consistent with their increased sensitivity and the coupling effects inherent in the elastic wave system.

\section{Concluding Remarks}
\label{sec_concluding}
In this paper, we have proposed a novel method for solving the inverse source problem for the elastic wave equation, focusing on the recovery of initial displacement and velocity fields from boundary measurements. A key contribution is the introduction of a new orthonormal basis in time, constructed from Legendre polynomials and exponential weights, which enables stable and accurate spectral representation of time-dependent functions. This basis facilitates a time-dimensional reduction strategy that transforms the original inverse problem into a sequence of spatial elliptic problems for the Fourier coefficients.

We have established a rigorous convergence theorem that ensures the reliability of the method under appropriate regularization and truncation conditions. The theoretical results confirm that the approximated space-time solution converges to the true minimal-norm solution as the noise level tends to zero. The convergence analysis is supported by a variational framework and compactness arguments that accommodate the ill-posedness of the inverse problem.

Numerical experiments in two spatial dimensions demonstrate the method's effectiveness in recovering both the qualitative structure and quantitative amplitude of the initial data. The reconstructions remain stable under noisy measurements, and the results show good agreement with the true profiles, even in challenging scenarios involving multiple inclusions or complex geometries.

 \section*{Acknowledgement}
 The work of Dang Duc Trong was supported by Vietnam National University (VNU-HCM) under grant number
 T2024-18-01.
The work of Loc Nguyen was partially supported by the National Science Foundation grant DMS-2208159.

\end{document}